\documentclass[11pt]{amsart}
\usepackage{a4wide}
\usepackage{mathrsfs}
\usepackage{amsfonts}
\usepackage{amsmath,extpfeil}
\usepackage{stmaryrd}
\usepackage{amssymb}
\usepackage{amsthm}
\usepackage{mathrsfs}
\usepackage{url}
\usepackage{amsfonts}
\usepackage{amscd}
\usepackage{indentfirst}
\usepackage{enumerate}
\usepackage{amsmath,amsfonts,amssymb,amsthm}
\usepackage{amsmath,amssymb,amsthm,amscd}
\usepackage{graphicx,mathrsfs}
\usepackage{appendix}
\usepackage[numbers,sort&compress]{natbib}
\usepackage{color}
\usepackage[colorlinks, linkcolor=blue, citecolor=blue]{hyperref}

\newcommand{\R}{\mathbb{R}}

\newtheorem{theorem}{Theorem}[section]
\newtheorem{lemma}[theorem]{Lemma}
\newtheorem{corollary}[theorem]{Corollary}
\newtheorem{proposition}[theorem]{Proposition}

\newtheorem{remark}[theorem]{Remark}
\newtheorem{conjecture}{Conjecture}
\numberwithin{equation}{section}

\newcommand{\ov}{\overline}

\newcommand{\dist}{\text{dist}}

\newcommand{\dis}{\displaystyle}
\newcommand{\lam}{\lambda}
\newcommand{\eps}{\varepsilon}
\newcommand{\pll}{\kern 0.56em/\kern -0.8em /\kern 0.56em}
\newcommand{\ome}{\Omega}
\newcommand{\beq}{\begin{equation}}
\newcommand{\eeq}{\end{equation}}

\begin{document}

\title[Free boundary regularity for the porous medium equation]
{Free boundary regularity for the porous medium equation}

\author[J. Liu]
{Jiakun Liu}
\address [Jiakun Liu]
	{School of Mathematics and Statistics,
	The University of Sydney,
	Camperdown, NSW 2006, AUSTRALIA}
\email{jiakun.liu@sydney.edu.au}

\author[J. Wang]{Junqing Wang}
\address [Junqing Wang]{Chern Institute of Mathematics, Nankai University, Tianjin 300071, CHINA}
\email{1120240020@mail.nankai.edu.cn}

\author[M. Zhou]{Maolin Zhou}
\address [Maolin Zhou]{Chern Institute of Mathematics, Nankai University, Tianjin 300071, CHINA}
\email{zhouml123@nankai.edu.cn}




\begin{abstract}
In this paper, we resolve the local cone conjecture of Kienzler, Koch, and V\'azquez concerning free boundary regularity for the porous medium equation. We prove that the local cone condition guarantees smoothness of the free boundary outside the initial support. The proof relies on new gradient estimates obtained through maximum principle and comparison arguments. 

As applications, we establish short-time smoothness of the free boundary for $C^1$ initial data with a non-degeneracy condition and global-in-time smoothness for continuous initial data with convex support and a weak non-degeneracy condition. 

Finally, we construct a counterexample showing that the assumptions in these two smoothness results are almost optimal and providing a negative answer to an open problem proposed by Daskalopoulos and Kenig.
\end{abstract}

\date{\today}
\maketitle

\baselineskip=16.4pt
\parskip=3pt



\section{Introduction}\label{s1}
We consider the Cauchy problem for the porous medium equation
\begin{equation}
	\label{pme u}
	\tag{1.1}
	\left\{
	\begin{aligned}
		&u_t=\Delta(u^m) &&\text{ in } \R^n\times (0,\infty),\\
		&u(\cdot,0)=u_0 &&\text{ in } \R^n,
	\end{aligned}
	\right.
\end{equation}
where \(n\ge 1\) and \(m>1\). Throughout, we assume that \(u_0\) is nonnegative, integrable, and compactly supported in $\R^n$.

As a basic model of nonlinear diffusion, \eqref{pme u} has been extensively studied since the 1950s.  We refer to \cite{V07} for a systematic presentation of the theory.  Compared with the heat equation, a distinctive feature of \eqref{pme u} is its degeneracy near \(u=0\).  This degeneracy gives rise to the phenomenon of \emph{finite propagation}: if the initial data \(u_0\) is compactly supported, then the solution \(u(\cdot,t)\) remains compactly supported for all \(t>0\).  Consequently, the support of \(u\) evolves with time.  We write
\[
\Omega:=\{(x,t)\in \R^n\times\R: u(x,t)>0\},\qquad
\Omega(t):=\{x\in \R^n: u(x,t)>0\},
\]
and denote the free boundary by $\Gamma:=\partial \Omega$, $\Gamma(t):=\partial \Omega(t)$. 
For continuous nonnegative solutions, the support of $u(\cdot,t)$ is $\overline{\Omega(t)}$.

Equation \eqref{pme u} also appears in a number of physical applications describing fluid flow or diffusion, most notably the flow of gas in a porous medium.  In this setting, \(u\) represents the gas density and \(v:=\frac{m}{m-1}u^{m-1}\) stands for its pressure.  The pressure \(v\) satisfies
\begin{equation}
	\label{pme v}
	\tag{1.2}
	\left\{
	\begin{aligned}
		&v_t=(m-1)v\Delta v+|\nabla v|^2 &&\text{ in } \ome,\\
		&v(\cdot,0)=v_0 &&\text{ in } \R^n,
	\end{aligned}
	\right.
\end{equation}
where \(v_0 = \frac{m}{m-1}u_0^{m-1}\).  
Since $u$ and $v$ have the same support, we use the same notation for their free boundaries. 
If \(v\) is smooth up to \(\Gamma\), equation \eqref{pme v} gives \(v_t = |\nabla v|^2\) on \(\Gamma\). Consequently, the normal speed of the free boundary at \(x\in\Gamma\) equals \(|\nabla v(x)|\).  
This relation is Darcy’s law and also appears in the Stefan problem. We therefore refer to it as the Stefan boundary condition.

Physical prediction suggests that, in an ideal environment, the diffusion of gas proceeds smoothly.  
For free boundary problems, regularity is a difficult and highly important issue.  
In this paper we therefore study mainly the regularity of the free boundary and of the solution \(v\) near it.  
Motivated by the interplay between degeneracy and the Stefan boundary condition, we focus from now on equation~\eqref{pme v}.

The regularity theory for \eqref{pme v} has undergone substantial development over many years.  
The theory is relatively complete in one space dimension; see \cite[Chapter 15]{V07}.  
In higher dimensions, the geometry of the free boundary introduces substantial additional difficulties. 
Caffarelli and Friedman~\cite{CF80} established H\"older regularity for solutions and free boundaries.  
In general, one cannot expect Lipschitz regularity at every positive time, because the focusing solutions of Aronson and Graveleau~\cite{AG93} shows that that $v$ is only H\"older continuous at the focusing time.

A natural question is whether solutions and their free boundaries become smooth after a sufficiently long time. 
Caffarelli, V\'azquez, and Wolanski~\cite{CVW87} proved that \(\Gamma\) and \(v\) are Lipschitz continuous for \(t>T_0\), where \(T_0\) is the first instant at which \(\Omega(T_0)\) covers a ball containing \(\Omega(0)\).  
They also introduced the non‑degeneracy condition \eqref{ndc}, which forces \(\Gamma(0)\) to move immediately at every point, and proved that under this condition \(\Gamma\) can be represented as \(t=S(x)\), with $S$ Lipschitz continuous.  
Caffarelli and Wolanski~\cite{CW90} subsequently improved the eventual regularity of \(\Gamma\) to \(C^{1,\alpha}\).

Daskalopoulos and Hamilton~\cite{DH98} used non-degeneracy to introduce coordinates that transform the free boundary problem into a degenerate equation on a fixed domain. By establishing Schauder estimates for the transformed equation, they obtained short-time $C^\infty$ regularity of \(\Gamma\) and \(v\) for non-degenerate initial data of class \(C^{2,\alpha}\).  
Koch~\cite{K99} developed H\"older estimates for the transformed problem, relaxed the initial regularity required for short-time smoothness to $C^{1,\alpha}$, and strengthened the eventual regularity theory to \(C^{\infty}\) smoothness and real analyticity. 
Daskalopoulos and Kenig~\cite{DK07} later asked whether the initial regularity required for short-time smoothness could be further relaxed to $C^{0,1}$.

A related question is to identify conditions that guarantee smoothness for all positive times. 
Daskalopoulos, Hamilton, and Lee~\cite{DHL01} introduced a root‑concavity condition for the initial data, namely the concavity of $\sqrt{v_0}$ on $\Omega(0)$.  
They proved that root‑concavity is preserved by the flow and used this property, together with a change of coordinates, to establish global-in-time \(C^{\infty}\) smoothness.  
This raises the question of whether root‑concavity can be replaced a weaker assumption, such as \(\alpha\)-concavity for some \(\alpha\in(0,\frac12)\) or mere convexity of \(\Omega(0)\), together with suitable non-degeneracy. (Note that \(\alpha\)-concavity of \(v_0\) implies convexity of \(\Omega(0)\)).  
Ishige and Salani~\cite{IS10} constructed an \(\alpha\)-concave initial data for some \(\alpha\in(0,\frac12)\) for which 
the \(\alpha\)-concavity was not preserved by the flow, and in fact 
\(\Omega(T)\) failed to be convex at some finite time \(T\).  
Their example shows that an extension of the root-concavity approach cannot rely on the preservation of convexity under these weaker assumptions. It leaves open the possibility of global smoothness despite the loss of convexity.

Kienzler, Koch, and V\'azquez~\cite{KKV18} developed an improvement of flatness approach and established eventual $C^\infty$ regularity of $v$ and $\Gamma$ without assuming non-degeneracy of the initial data. 
In that paper they formulated the following conjecture, which we call the \emph{local cone conjecture}.

\begin{conjecture}[\cite{KKV18}]
	\label{conj KKV}
Suppose that $(0,0)\in\Gamma$ and that $v$ is monotone in every spatial direction belonging to an open cone throughout $B_1\times[-1,0]$. Then, locally in $B_1\times(-1,0]$ and outside the initial support $\overline{\Omega(-1)}$, the pressure $v$ is $C^\infty$ up to the free boundary and $\nabla v\neq0$ on the free boundary.
\end{conjecture}

In this paper, we resolve this conjecture by establishing a local criterion that prevents the formation 
of singularities and guarantees smooth evolution of the free boundary.  
We first recall the \textbf{local cone condition} from \cite{KKV18}.
For $\mu\in\mathbb{S}^{n-1}$ and $\theta\in(0,\pi/2]$, let
		\begin{align*}
			\mathcal{C}_{\mu,\theta}:=\left\{V\in\R^n \, :\, V\cdot \mu\geq |V|\cos\theta\right\}.
		\end{align*}
We say that $v$ satisfies the local cone condition in $B_1\times[-1,0]$ if there exist such $\mu$ and $\theta$ for which
		\begin{align*}
			v(x+V,t)\geq v(x,t),
		\end{align*}
whenever \(t\in[-1,0]\), \(V\in\mathcal{C}_{\mu,\theta}\), and spatial points \(x\), \(x+V\in B_1\).

Our first main result is the following.
\begin{theorem}
	\label{thm 1.1}
	Let \(u\) be a continuous weak solution of \eqref{pme u} in the local coordinates as above, and let \(v=\frac{m}{m-1}u^{m-1}\). Suppose that $v$ satisfies the local cone condition in $B_1\times[-1,0]$. Then, for any \((q,t_0)\in\Gamma\cap \left(B_1\times(-1,0]\right)\) with \(q\notin \overline{\ome(-1)}\), there exists \(\delta_0>0\) such that \(v\) is smooth up to \(\Gamma\), and \(\Gamma\) is a smooth hypersurface in \(B_{\delta_0}(q)\times[t_0-\delta_0,t_0]\). Moreover, \(\nabla v\neq0\) on \(\Gamma \cap \left( B_{\delta_0}(q)\times[t_0-\delta_0,t_0]\right)\). 
\end{theorem}

\begin{remark}
	\label{remark 1}
\emph{
After a rotation, we may assume that \(\mu=e_n\). Writing \(L=\cot \theta<\infty\) and \(x=(x',x_n)\), with \(x'=(x_1,\cdots,x_{n-1})\),  the local cone condition becomes  
\begin{align*}
	v\left(x'+\xi,x_n+\eta,t\right)\geq v(x',x_n,t)\qquad\text{whenever } \eta\geq L|\xi|,
\end{align*}
provided that \(t\in[-1,0]\) and both \((x',x_n)\) and \((x'+\xi,x_n+\eta) \) lie in \(B_1\).
}
\end{remark}

\begin{remark}
\emph{
Gradient estimates near the free boundary are a key ingredient in the proof of Theorem~\ref{thm 1.1}. 
A central step is to establish an estimate of the form
	\begin{equation}
		\label{kzcond}
		\tag{1.3}
		v_t\leq C\bigl(v+|\nabla v|\bigr)
	\end{equation}
under the local cone condition. 
Such an estimate plays an important role in the eventual regularity theory of Caffarelli, V\'azquez, and Wolanski~\cite{CVW87}. 
In \cite{KZ01}, Kim and Zhang imposed a related bound on the time derivative in their local regularity analysis. 
Deriving \eqref{kzcond} directly from local cone monotonicity is the main technical difficulty of the present paper.
}

\emph{
The estimate \eqref{kzcond} indeed implies a local upper bound for $|\nabla v|$.
Kim and Zhang~\cite[Theorem~1.3]{KZ01} further established local linear growth of $v$ at the free boundary for the porous medium equation with a smooth drift. 
In the absence of drift, this provides a partial answer to the local cone conjecture. 
In our approach, establishing \eqref{kzcond} and the resulting local upper bounds on $|\nabla v|$ and $|v_t|$ allows us to simplify the comparison arguments used to prove linear growth at the free boundary. Rescaling and Harnack's inequality then yield a positive lower bound for a directional derivative of $v$ in a neighbourhood of the free boundary. 
Together, these upper and lower gradient bounds enable us to apply the smoothness criterion recalled below. 
}
\end{remark}

As our first application of Theorem~\ref{thm 1.1}, we improve the short-time regularity results of Daskalopoulos and Hamilton~\cite{DH98} and Koch~\cite{K99}. We show that $v_0\in C^1(\ov{\ome(0)})$, together with non-degeneracy, suffices for smoothness at all sufficiently small positive time. 
\begin{theorem}
	\label{short-time}
	Suppose that $v_0\in C^1(\overline{\Omega(0)})$ satisfies the non-degeneracy condition
	\begin{equation}
		v_0+|\nabla v_0|\ge c>0 \qquad\text{on }\overline{\Omega(0)}
		\label{ndc}
		\tag{1.4}
	\end{equation}
	for some constant $c>0$. Then there exists $\varepsilon_0>0$, depending only on $n,m,v_0,c$, such that $v$ is smooth up to $\Gamma$ and $\Gamma$ is a smooth hypersurface for $0<t <\varepsilon_0$.
\end{theorem}
The proof of Theorem~\ref{short-time} is local and yields a corresponding local regularity statement when the hypotheses hold near a portion of the initial free boundary. Our main observation is that for $C^1$ initial data satisfying~\eqref{ndc}, we use a blow-up argument to show that the local cone condition is preserved in a short time interval. Condition~\eqref{ndc} ensures immediate expansion of the support set. Theorem~\ref{thm 1.1} therefore applies. More generally, the same strategy applies whenever the required local cone monotonicity can be shown to persist.

Our second application concerns smoothness for all positive times. 
We replace the root-concavity assumption in the work of Daskalopoulos, Hamilton, and Lee~\cite{DHL01} by convexity of $\Omega(0)$ and a weak non‑degeneracy condition. $v_0$ is assumed only to be continuous, and the argument does not require the evolving support $\Omega(t)$ to remain convex. 
\begin{theorem}
	\label{global-time}
	Suppose that $\Omega(0)$ is bounded and convex, and that $v_0\in C(\R^n)$ satisfies the weak non-degeneracy condition
	\begin{equation}
		v_0(x) \geq c\left[\operatorname{dist}(x,\Gamma(0))\right]^{\beta}\qquad\text{for $x\in\ome(0)$ sufficiently close to $\Gamma(0)$},
		\label{wndc}
		\tag{1.5}
	\end{equation}
	where $c>0$ and $\beta\in(0,2)$. Then $v$ is smooth up to $\Gamma$ and $\Gamma$ is a smooth hypersurface for every $t>0$.
\end{theorem}

For the proof of Theorem~\ref{global-time}, the convexity of $\Omega(0)$ gives local cone monotonicity outside the initial support by Alexandrov reflection, and~\eqref{wndc} ensures that the free boundary moves immediately. Theorem~\ref{global-time} thus follows directly from Theorem~\ref{thm 1.1}, even if $\Omega(t)$ loses convexity, as in the example of Ishige and Salani~\cite{IS10}.

\begin{remark}
	\emph{
		Our route to Theorem~\ref{thm 1.1} began with a proof of global-in-time smoothness for convex $\Omega(0)$ under the stronger
		assumptions $v_0\in C^3(\overline{\Omega(0)})$ and the non-degeneracy condition~\eqref{ndc}.
		In that argument, for $x\in\Omega(t)\setminus\overline{\Omega(0)}$, let $d(x)=\operatorname{dist}(x,\Omega(0))$ and let $\nu=-\nabla d$, so that $\nu$ point towards the initial support. Using the linearisation of~\eqref{pme v} and the maximum principle, we obtained the estimate $v_t\leq C(v+\nu\cdot\nabla v)$, and consequently, an upper bound for $|\nabla v|$. 
		We subsequently realised that the key ideas underlying this argument could be
		extended to the more general setting of local cone monotonicity.
		Developing this observation led us to establish
		Theorem~\ref{thm 1.1}, which in turn yields global-in-time smoothness
		under the weaker hypotheses of Theorem~\ref{global-time}.
	}
	
	\emph{
		If the weak non-degeneracy condition~\eqref{wndc} is omitted, the conclusion still holds locally at free boundary points outside $\Omega(0)$. More precisely, under the remaining hypotheses of Theorem~\ref{global-time}, for every $(q,t_0)\in\Gamma\cap \left(\R^n\times(0,\infty)\right)$ with $q\notin \overline{\ome(0)}$, there exists $\delta_0$ such that $v$ is smooth up to $\Gamma$ and $\Gamma$ is smooth in $B_{\delta_0}(q)\times[t_0-\delta_0,t_0]$.
	}
\end{remark}

Finally, we construct a counterexample that demonstrates the limitations of weakening the hypotheses in our two applications.
First, in higher dimensions, short-time smoothness can fail even for non-degenerate initial data
$v_0\in C^1\left(\ome(0)\right)\cap C^{0,1}(\R^n)$
with a Lipschitz initial domain $\ome(0)$. Thus, $C^1$ regularity up to the initial boundary cannot, in general, be replaced by Lipschitz continuity. This gives a negative answer to the question posed by Daskalopoulos and Kenig~\cite{DK07} and shows that the initial regularity assumption in Theorem~\ref{short-time} is nearly optimal. 
Second, the counterexample shows that star-shapedness of $\ome(0)$ does not suffice for global-in-time smoothness. In particular, the convexity assumption in Theorem~\ref{global-time} cannot, in general, be weakened to star-shapedness.

For comparison, Jin, Ros-Oton, and Xiong~\cite{JRX24} proved the optimal long-time regularity of the nonnegative solution to the porous medium equation in smooth bounded domains with zero Dirichlet boundary conditions. They showed that, after a finite waiting time, $u^m$ is $C^{2,1/m}$ in space up to the fixed boundary and $C^\infty$ in time, and that the spatial exponent is optimal. Their result highlights the different boundary behaviour of the Dirichlet and free boundary problems. In the latter, the pressure $v$, proportional to $u^{m-1}$, can be smooth up to the moving boundary, as in the results above.

The paper is organised as follows. Section~\ref{s2} introduces notation and recalls preliminary results for the porous medium equation. In Section~\ref{s3}, we construct smooth approximations that preserve the local cone condition. Section~\ref{s4} establishes the local upper gradient estimates for smooth solutions, which form the main technical part of the paper. In Section~\ref{s5}, we use these estimates to derive the lower gradient bounds. Finally, Section~\ref{s6} contains the proofs of the main theorems and the construction of the counterexample.

\vskip5pt
\noindent{\bf Acknowledgment:} 
We are grateful to Professor Ki-Ahm Lee for many stimulating discussions and, in particular, for drawing our attention to the question posed by Daskalopoulos and Kenig in \cite{DK07}.

\noindent{\bf AI disclosure:}
The research project in this paper was initiated by the third author in 2016. The main part of the manuscript (Section \ref{s4} and Theorem \ref{global-time}) was completed jointly by the authors in 2024.  
Subsequently, the authors used ChatGPT (OpenAI) to assist with technical computations in Section \ref{s3}. 
(The construction of approximate solutions in that section is independent of their earlier proof of global-in-time smoothness for convex initial support.)

The authors originally constructed two separate counterexamples: 
one showing that Lipschitz regularity of $v_0$ is insufficient for the short-time smoothness in Theorem~\ref{short-time}, and the other showing that star-shapedness of $\Omega(0)$ is insufficient for the global-in-time smoothness in Theorem~\ref{global-time}. With AI assistance, they subsequently constructed the single example in Proposition~\ref{counterexample}, which establishes both conclusions.

OpenAI tools were also used for language editing, LaTeX formatting, and the organisation of the presentation. 
The authors have reviewed the complete manuscript and take sole responsibility for all mathematical statements and proofs, the claims of originality, and the relation of the results to the existing literature.



\section{Preliminaries}\label{s2}
In this section, for the reader’s convenience, we collect here some notation and preliminary results for the porous medium equation, with particular emphasis on the local coordinate change. We denote by $\Omega$ the support of $v$: $\Omega:=\left\{\left(x,t\right)\in\R^n\times\R:v\left(x,t\right)>0\right\}$, and set $\Omega(t):=\left\{x\in\R^n:v\left(x,t\right)>0\right\}$. The boundary are denoted by $\Gamma:=\partial \Omega$ and $\Gamma(t):=\partial \Omega(t)$. We write $B_r(x)$ for the open ball $\{y\in\R^n : |y-x|<r\}$ and set $B_r=B_r(0)$. We denote by $\langle \cdot,\cdot\rangle$ the standard scalar product on $\R^n$ and by $\operatorname{dist}(\cdot,\cdot)$ the associated distance. We write $\pi_{x_i}(x)=x_i$ for the projection onto the coordinate $x_i$. We denote by $x\otimes y=xy^T$ for two vectors. We use notation $(f)_+=\max\left\{f,0\right\}$ for the positive part of function $f$. Finally, $C^{k,\alpha}$ stands for the H\"older space with exponent $\alpha\in(0,1)$.

We first state the classical Aronson-B\'enilan estimate.
\begin{proposition}[\cite{AB79}]
	\label{AB prop}
	For equation \eqref{pme v}, the following Aronson-B\'enilan inequality holds in the sense of distributions:
	\begin{align*}
		\Delta v\geq -\frac{n}{\left(n(m-1)+2\right)t}\quad\mbox{on $\R^n\times\left(0,\infty\right)$}.
	\end{align*}
	Moreover, for any $x\in\R^n$ and $0<s\leq t$
	\begin{align*}
		v(x,t)\geq \left(\frac{s}{t}\right)^{\frac{n(m-1)}{n(m-1)+2}}v(x,s).
	\end{align*}
\end{proposition}
Proposition \ref{AB prop} holds on $\R^n$ for Cauchy problem. We can give a local version that is useful for our local question using similar proof.
\begin{proposition}
	\label{local AB prop}
	For equation \eqref{pme v} on $B_{2R_0}\times(0,T_0]$, suppose $v\in[0,M]$ for some constant $M>0$, then there is constant $K_0=K_0(n,m)$ such that in the sense of distributions
	\begin{align*}
		\Delta v\geq -K_0\left(\frac{1}{t}+\frac{M}{R_0^2}\right)\quad\text{on $B_{R_0}\times(0,T_0]$}.
	\end{align*}
	Moreover, fixed $\delta>0$, there is constant $\dis K_{\delta}$, depending only on $n$, $m$, $\delta$, $R_0$, $M$, such that 
	\begin{align*}
		\Delta v\geq -K_{\delta} \quad \text{on $B_{R_0}\times[\delta,T_0]$}\quad\text{in the sense of distributions},
	\end{align*}
	and that for any $x\in B_{R_0}$ and $\delta\leq s\leq t\leq T_0$,
	\begin{align*}
	\qquad v(x,t)\geq e^{-(m-1)K_{\delta}(t-s)} v(x,s).
	\end{align*}
\end{proposition}

Then we state two results in \cite{CF80}. The first shows that  the free boundary can not enter a fixed area without enough density. The second shows that once the free boundary start to move, it will not stop.
\begin{proposition}[\cite{CF80}]
	\label{CF prop}
	Suppose $\Delta v\geq -K$ in the sense of distributions on $B_{R_0}\times\left(0,\infty\right)$. There exists constants $\sigma_0,\ c_0>0$, depending only on $m$, $n$, $K$, such that for any $x_0\in B_{R_0}$, $t_0\in\left(0,\infty\right)$, $R\in(0,\operatorname{dist}(x_0, B_{R_0}))$, $0<\sigma<\sigma_0$, if
	\begin{align*}
		v(x,t_0)\equiv0\quad\text{on $B_R(x_0)$},\qquad \frac{1}{|B_R(x_0)|}\int_{B_R(x_0)}v(x,t_0+\sigma)dx\leq \frac{c_0 R^2}{\sigma},
	\end{align*}
	then
	\begin{align*}
		v(x,t_0+\sigma)\equiv0\quad\text{on $B_{R/6}(x_0)$}.
	\end{align*}
\end{proposition}

\begin{theorem}[\cite{CF80}]
	\label{CF thm}
	For any boundary point $(x_0,t_0)\in\Gamma$, either every point of $\left\{x=x_0,0\leq t<t_0\right\}$ belongs to $\Gamma$ or no point of $\left\{x=x_0,0\leq t<t_0\right\}$ belongs to $\Gamma$.
\end{theorem}

The short-time smoothness for $C^{1,\alpha}$ initial data is stated as following and the local version can be found in Theorem 5.5.1 in \cite{K99}. It's first proved in \cite{DH98} for $C^{2,\alpha}$ initial data and improved in \cite{K99}.
\begin{theorem}[\cite{DH98,K99}]
	\label{DH thm}
	Let $v_0\in C^{1,\alpha}\left(\overline{\Omega(0)}\right)$ satisfy the condition \eqref{ndc} for some $\alpha>0$. Then there exists a constant $\varepsilon_0>0$, depending only on $n$, $m$, $v_0$, $c$ such that $v(x,t)\in C^{\infty}\left(\overline{\Omega}\right)$ and $\Gamma$ is $C^{\infty}$ smooth for $t\in\left(0,\varepsilon_0\right)$. 
\end{theorem}

Finally, we state a $C^{\infty}$-regularity result proved in \cite{K99}. It shows that as long as $v$ is linear growth with respect to a certain direction near the free boundary, then $v$ is smooth up to the free boundary. The proof of this $C^{\infty}$-regularity result uses the local coordinate change in \cite{DH98} and it also can be found in the proof of Theorem 5.1 in \cite{DHL01}. 
\begin{theorem}[\cite{K99,DHL01}]
	\label{thm tool}
	Let $(x_0,t_0)\in\Gamma$ be a boundary point. Assume that there exists constants $\delta_0,C_0>0$ and a unit vector $\mu$ such that
	\begin{align*}
		\left|\nabla v\left(x,t\right)\right|\leq C_0^{-1},\qquad D_{\mu}v\left(x,t\right)\geq C_0\quad\mbox{for all $(x,t)\in U_{\delta_0}(x_0,t_0)$},
	\end{align*}
	where
	\begin{align*}
		U_{\delta_0}(x_0,t_0):=\left\{(x,t):x\in\Omega(t)\cap B_{\delta_0}(x_0),\ t\in\left[t_0-\delta_0,t_0\right]\right\}
	\end{align*}
	Then $v(x,t)\in C^{\infty}(\overline{\Omega}\cap U_{\delta_0/2}(x_0,t_0))$ and $\Gamma\cap U_{\delta_0/2}\left(x_0,t_0\right)$ is smooth.
\end{theorem}



\section{Smooth approximation preserving local cone}\label{s3}

In this section, we prove a smooth approximation that preserves the local cone condition. We want to construct smooth Cauchy-Dirichlet data that preserve the local cone condition. The initial data can be obtain by a spatial modified. However, the lateral boundary is hard to construct. Our main idea is to choose a cube and identity it as the torus
on which the lateral boundary vanishes. Then we obtain a sequence of smooth $v_j$ and we show $v_j\to v$ on the local domain. Since we want to introduce the torus, our solution need to be periodized. In order to preserve the local cone condition in the periodization, we use the level set function $f(x',z,t)$. To make true $f$ is well-defined, we first prove the following lemma.

Using the setting in Remark \ref{remark 1}, the local cone condition gives a representation of the free boundary. Precisely, there exists a Lipschitz function $f(x',t)$ with Lipschitz constant $L$ such that
\begin{align*}
	\Omega=\left\{(x',x_n,t):x_n>f(x',t)\right\}\quad\text{on $B_1\times[-1,0]$}.
\end{align*}

\begin{lemma}\label{lem 3.1}
	 For any small $\delta>0$, on $B_{1-\delta}\times[-1+\delta,0]$
	\begin{align*}
		D_{x_n}v>0,\qquad |D_{x'}v|\leq L D_{x_n}v\quad\text{on $\ome$}.
	\end{align*}
	Fix a small constant $Z_0>0$ and for all $0<z\leq Z_0$, equation $v\left(x',f\left(x',z,t\right),t\right)=z$ define a unique smooth function $f$ with
		\begin{align*}
		D_z f=\frac{1}{D_{x_n}v},\qquad D_{x'}f=-\frac{D_{x'}v}{D_{x_n}v},\qquad f_t=-\frac{v_t}{D_{x_n}v},\qquad\left|D_{x'}f\right|\leq L.
	\end{align*}
	It extends continuously to $z=0$, with $f(x',0,t)=f(x',t)$.
\end{lemma}
\begin{proof}
	Smoothness in the positive set and local cone condition give $D_{x_n} v\geq0$. For every unit vector $e\in\R^{n-1}$ and small $\eps>0$, consider vector $\eps(e,L)$ and $\eps(-e,L)$. Then local cone condition gives
	\begin{align*}
		v(x'+\eps e, x_n+\eps L,t)\geq v(x,t)\qquad v(x'-\eps e, x_n+\eps L,t)\geq v(x,t).
	\end{align*}
	Differentiation with respect to $\eps$ at $\eps=0$ in the positive set gives $|D_{x'} v\cdot e|\leq L D_{x_n}v.$
	Taking the supremum in $e$ proves
	\begin{align*}
		\label{3.1}
		\tag{3.1}
		|D_{x'}v|\leq L D_{x_n}v.
	\end{align*}
	
	Fix a positive point $(X,T)$ on $B_{1-\delta}\times[-1+\delta,0]$. Choose $-1<T'<T$ sufficiently close that $v(X,T')>0$. For $\eps>0$ the fundamental theorem of calculus gives
	\begin{align*}
		\int_{f(X',T')+\eps}^{X_n} D_{x_n} v\left(X',x_n,T'\right)dx_n=v\left(X',X_n,T'\right)-v\left(X',f(X',T')+\eps,T'\right).
	\end{align*}
	The second term tends to zero as $\eps\to0$. Since $D_{x_n} v\geq0$, monotone convergence yields
	\begin{align*}
	\int_{f(X',T')}^{X_n} D_{x_n} v\left(X',x_n,T'\right)dx_n=v\left(X',X_n,T'\right)>0.
	\end{align*}
	There is therefore a $Y_n\in\left(f(X',T'),X_n\right)$ with $D_{x_n}v\left(X',Y_n,T'\right)>0$. On the compact segment $\left\{X'\right\}\times[Y_n,X_n]$, $v(\cdot,T')$ has positive minimun. Proposition \ref{local AB prop} then shows $v$ has positive minimum on this segment for all $t\in[T',T]$. Continuity and compactness give a spatial neighborhood of this segment on which $v$ has a uniform positive lower bound for all $t\in[T',T]$. In this new cylinder, interior parabolic estimates infer $v\in C^{\infty}$ with uniform bound. Set $p=D_{x_n} v$. In above cylinder $p$ solves the following equation
	\begin{align*}
		p_t=(m-1)v\Delta p+2\nabla v\nabla p+(m-1)\Delta v p.
	\end{align*}
	Since $p\geq0$ and solves a uniform parabolic equation, Harnack inequality gives a uniform constant $C>0$ such that
	\begin{align*}
		p(X,T)\geq C p(X',Y_n,T')>0.
	\end{align*}
	
	Choose $Z_0$ smaller enough such that $\left\{v=Z_0\right\}$ stays on $B_{1-\delta}\times[-1+\delta,0]$. For any $0<z\leq Z_0$, since $D_{x_n}v>0$, the implicit function theorem gives a function $x_n=f(x',z,t)$ satisfies $v\left(x',f(x',z,t),t\right)=z$. Differentiation of this equation gives
	\begin{align*}
		D_z f=\frac{1}{D_{x_n}v},\qquad D_{x'}f=-\frac{D_{x'}v}{D_{x_n}v},\qquad f_t=-\frac{v_t}{D_{x_n}v}.
	\end{align*}
	The smoothness of $v$ implies the smoothness of $f$  and \eqref{3.1} gives $\left|D_{x'}f\right|\leq L$ in the positive set. The continuity at $z=0$ follows from the continuity of $v$.
\end{proof}

Now we fixed our target domain. Choose a suitable ball  $B_{0}^{n-1}\Subset \pi_{x'}B_{1-\delta}$ and a constant $F_-$ satisfying
\begin{align*}
	-1+\delta \ll F_-<\min \left\{f(x',t):(x',t)\in B_{0}^{n-1}\times\left[-1+\delta,0\right]\right\},
\end{align*}
where the suitable of $B^{n-1}_0$ means the choice of $F_-$ is well-defined. Here and in the following, we use the superscript $n-1$ to represent the ball lies in the $x'$-space. It is easy to see that $\left\{x_n=F_-\right\}$ lies strictly below every zero set and $v=0$ in a strip of fixed positive width above $\left\{x_n=F_-\right\}$ . 

\begin{proposition}\label{smooth appro}
	For any $\delta>0$ in Lemma \ref{lem 3.1}. There are positive smooth functions $u_j$,  a ball $B_*^{n-1}\Subset B_0^{n-1}$, constants $Z_*=\frac{m}{m-1}M_*^{m-1}>0$, $L_*\in\left(0,\infty\right)$, independent of $j$, and positive numbers $\gamma_j\to 0$, with the following properties.
	\begin{enumerate}
		\item On domain $\left\{(x',x_n,t):x\in B_*^{n-1},\ F_-\leq x_n\leq f(x',Z_*,t),\ t\in\left[-1+\delta,0\right]\right\}$,
		\begin{align*}
			&(u_j)_t=\Delta\left(u_j^m\right),\qquad \gamma_j\leq u_j\leq M_*,\qquad |D_{x'}u_j|\leq L_* D_{x_n}u_j,\\
			&D_{x_n} u_j>0 \quad \text{in the interior of above domain}.
		\end{align*}
		On the bottom boundary $\left\{(x',x_n,t):x\in B_*^{n-1},\  x_n=F_-,\  t\in\left[-1+\delta,0\right]\right\}$,
		\begin{align*}
			u_j=\gamma_j.
		\end{align*}
		On the top boundary $\left\{(x',x_n,t):x\in B_*^{n-1},\  x_n=f(x',Z_*,t),\  t\in\left[-1+\delta,0\right]\right\}$,
		\begin{align*}
			u_j=M_*.
		\end{align*}
		\item For fixed $j$, the corresponding level set functions $f_j$ are well-defined and smooth in the interior of above domain. And for fixed $0<Z<Z_*$ and large $j$, all spatial and time derivatives of $f_j(x',Z,t)$ are uniformly bounded in $j$.
		\item On above domain, $u_j\to u$ and $\dis v_j=\frac{m}{m-1}u_j^{m-1}\to v$ uniformly. Additionally, on the compact of $\Omega$, the convergence holds with all spatial and time derivatives.
	\end{enumerate}
\end{proposition}

Choose nested balls 
\begin{align*}
	B_*^{n-1}\Subset B_1^{n-1}\Subset B_2^{n-1}\Subset B_3^{n-1} \Subset B_4^{n-1}\Subset B_5^{n-1}\Subset B_0^{n-1}.
\end{align*}
These balls leaves enough area for the cut-off function in the later. Also we choose $B_{*}^{n-1}$ as large as possible for our later use in Section \ref{s6}. Then we state our main step of the proof.

First, we identify a large cube containing $B_0^{n-1}$ with a flat torus $\mathbb{T}^{n-1}$ and periodize $u$. In order to periodize $u$, we first periodize $f$. Let $\chi_0$ be the smooth cut-off function with
\begin{align*}
	\chi_0=1\quad\text{on $B_5^{n-1}$},\qquad \chi_0=0\quad\text{on $\left(B_0^{n-1}\right)^c$}.
\end{align*}
Choose a linear function $f_0(z)=d_0+d_1 z$ with constant $d_0>F_-$ and $d_1>0$. Then define
\begin{align*}
	\widehat{f}(x',z,t)=f_0(z)+\chi_0(x')\left(f(x',z,t)-f_0(z)\right)\quad\text{for $x'\in\mathbb{T}^{n-1}$, $z\in\left[0,Z_0\right]$, $t\in\left[-1+\delta,0\right]$},
\end{align*}
where $Z_0$ is given in Lemma \ref{lem 3.1}. Clearly, $\widehat{f}$ is a $x'$-periodic function on $\mathbb{T}^{n-1}$. The $x'$-derivative of $\widehat{f}$ satisfies
\begin{align*}
	|D_{x'}\widehat{f}|\leq |D_{x'}f|+|D\chi_0||f-f_0|\leq L_0<\infty.
\end{align*}
Thus, $\widehat{f}$ preserve local cone condition for a new constant $L_0$.

For $z>0$,
\begin{align*}
	\widehat{f}_z=\chi_0 f_z+(1-\chi_0)d_1>0.
\end{align*}
Then we can define $\widehat{u}$ by inverting $\widehat{f}$ between $0$ and $Z_*$, where $Z_*<Z$ is a smaller constant. Precisely, define $\widehat{u}$ be the unique function solve $\widehat{f}(x',\frac{m}{m-1}\widehat{u}(x',x_n,t)^{m-1},t)=x_n$, for $\widehat{f}(x',0,t)\leq x_n\leq \widehat{f}(x',Z_*,t)$. Set
\begin{align*}
	&\widehat{u}(x',x_n,t)=0\quad\text{for $x_n\leq \widehat{f}(x',0,t)$},\\
	&\widehat{u}(x',x_n,t)=M_*=\left(\frac{m-1}{m}Z_*\right)^{1/(m-1)}\quad\text{for $x_n\geq \widehat{f}(x',Z_*,t)$}
\end{align*}
Therefore, $\widehat{u}$ is a continuous function on $\mathbb{T}^{n-1}\times \R\times\left[-1+\delta,0\right]$ with values in $[0, M_*]$. The continuity of $\widehat{u}$ follows from the continuity of $\widehat{f}$. Since $\left|D_{x'}\widehat{f}\right|\leq L_0$, we can show that local cone condition holds for $\widehat{u}$ with constant $L_0$. Indeed, for $x_n\leq \widehat{f}(x',0,t)$, 
\begin{align*}
	\widehat{u}(x+V,t)\geq 0=\widehat{u}(x,t).
\end{align*}
For $x_n\geq \widehat{f}(x',Z_*,t)$,
\begin{align*}
	\widehat{u}(x+V,t)=M_*=\widehat{u}(x,t)
\end{align*}
 For $\widehat{f}(x',0,t)\leq x_n\leq \widehat{f}(x',Z_*,t)$, let $\widehat{u}(x',x_n,t)=\left(\frac{m-1}{m}p_0\right)^{1/(m-1)}$ and $\eta\geq L_0|\xi|$, then
\begin{align*}
	x_n+\eta\geq \widehat{f}(x',p_0,t)+L_0|\xi|\geq\widehat{f}(x+\xi,p_0,t).
\end{align*}
Therefore
\begin{align*}
	\widehat{u}(x'+\xi,x_n+\eta,t)\geq \left(\frac{m-1}{m}p_0\right)^{1/(m-1)}= \widehat{u}(x',x_n,t).
\end{align*}
Consequently, we periodize $u$ and preserve the local cone condition. Moreover, $\widehat{u}=0$ in a strip of fixed positive width above $\left\{x_n=F_-\right\}$, because $d_0>F_-$. Also, it agrees with $u$ for $
x'\in B_5^{n-1} $ and $F_-< x_n<\widehat{f}(x',Z_*,t)$.

Second, we modify $\widehat{u}$ to obtain the initial data. Denote
\begin{align*}
	D(t)=\left\{(x',x_n):x'\in\mathbb{T}^{n-1},\ F_-<x_n<\widehat{f}(x',Z_*,t)\right\}.
\end{align*}
Let $\gamma_j\downarrow 0$. Choose $\delta_j\downarrow0$ and $\rho_j>0$ with 
\begin{align*}
	\label{3.2}
	\tag{3.2}
	2(L_0+1)\rho_j<\delta_j.
\end{align*}
Let $\varrho_j$ be a nonnegative smooth spatial convolution kernel, supported in $\left\{|y|<\rho_j\right\}$, with periodic convolution in $x'$. Define 
\begin{align*}
	g_j(x)=\gamma_j^{}+(1-\gamma_j/M_*)\int \varrho_j(y)\widehat{u}(x-y+\delta_je_n,-1+\delta)dy\quad\text{on $D(-1+\delta)$}.
\end{align*}
By definition, smooth function $g_j$ satisfies $\gamma_j\leq g_j\leq M_*$ and converges uniformly to $\widehat{u}$. It preserves the local cone condition with same constant $L_0$.

Moreover, for all $|y|\leq \rho_j$, the vector $\delta_j e_n-y$ belong to the cone in local cone condition by \eqref{3.2}. Therefore
\begin{align*}
	\label{3.3}
	\tag{3.3}
	g_j(x)\geq \gamma_j+(1-\gamma_j/M_*)\widehat{u}(x,-1+\delta)\geq \widehat{u}(x,-1+\delta).
\end{align*}
Recall that $\widehat{u}(x,-1+\delta)=M_*$ for  $x_n\geq \widehat{f}(x',Z_*,-1+\delta)$. Then, \eqref{3.2} gives 
\begin{align*}
	g_j(x)=M_*\quad\text{for $\widehat{f}(x',Z_*,-1+\delta)-\left(\delta_j-(L_0+1)\rho_j\right)\leq x_n\leq \widehat{f}(x',Z_*,-1+\delta)$}.
\end{align*}
Similarly, $g_j=\gamma_j$ in a strip of fixed positive width above $\left\{x_n=F_-\right\}$.

Third, we approximate $u$ by a sequence of smooth functions $p_j$ on the local domain. The smoothness of $p_j$ will later be used to establish the desired convergence $u_j\to u$ on the local domain.

To construct the smooth $p_j$ on the local domain, we smooth this domain first. Take a ball $B^{n-1}$ with $B_5^{n-1}\Subset B^{n-1}\Subset\left\{\chi_0=1\right\}$. This is possible because we can let $\chi_0=1$ near $B_5^{n-1}$. Thus $\widehat{u}=u$ for $x'\in B^{n-1}$. Consider the domain $\left\{x'\in B^{n-1},\ x_n\in\left(F_-,\widehat{f}(x',Z_*,t)\right)\right\}$. Since $\widehat{f}(x',Z_*,t)$ is space-time smooth, we can smooth above domain at the spatial corner, outside $B_5^{n-1}$, to obtain a space-time smooth domain $D_{\text{loc}}(t)$. 

Consider the problem
\begin{equation}
	\notag
	\left\{
	\begin{aligned}
		&(p_j)_t=\Delta(p_j^m), &&(x,t)\in D_{\text{loc}}(t)\times [-1+\delta,0],\\
		&p_j(x,-1+\delta)=g_j(x), &&x\in D_{\text{loc}}(-1+\delta),\\
		&p_j(x,t)=b_j(x,t), &&(x,t)\in  \partial D_{\text{loc}}(t)\times[-1+\delta,0].
	\end{aligned}
	\right.
\end{equation}
Choose continuous lateral data $b_j$ satisfying 
\begin{align*}
	\gamma_j\leq b_j\leq M_*,\qquad b_j\geq u,\qquad  b_j(x,-1+\delta)=g_j(x),\qquad |b_j-u|\to 0.
\end{align*}
Additionally, we require $b_j=\gamma_j$ for $x_n=F_-$ and $b_j=M_*$ for $x_n=\widehat{f}(x',Z_*,t)$.

In fact, we can give an explicit construction. Let $\Phi_t$ be a smooth family of diffeomorphisms from $D_{\text{loc}}(-1+\delta)$ onto $D_{\text{loc}}(t)$, transporting the top and bottom boundary onto each other. Set $G_j(x,t)=g_j(\Phi_t^{-1}x)$ and $q_j=\gamma_j+(1-\gamma_j/M_*)u$. Choose a smooth cut-off function $0\leq \theta_j(t)\leq 1$ satisfying $\theta_j(-1+\delta)=1$ and  $\theta_j(t)=0$ for $t\geq -1+\delta+\eta_j$, where $\eta_j\downarrow0$. Define
\begin{align*}
	b_j=q_j+\theta_j\left(G_j-q_j\right)_+.
\end{align*}
Since $\gamma_j\leq G_j,q_j\leq M_*$, we obtain $\gamma_j\leq b_j\leq M_*$. By definition, $b_j\geq q_j\geq u$. At $t=-1+\delta$, inequality \eqref{3.3} gives $G_j=g_j\geq q_j$, hence $b_j=g_j$. For $x_n=F_-$, $G_j=q_j=b_j=\gamma_j$ and for $x_n=\widehat{f}(x',Z_*,t)$, $G_j=q_j=b_j=M_*$. The uniform convergence follows easily from the uniform convergence of $g_j$ to $u$.

 For each $j$, above Cauchy-Dirichlet problem on a smooth moving domain admits a unique continuous solution $p_j$. The comparison principle gives
 \begin{align*}
 	\gamma_j\leq p_j\leq M_*,\qquad p_j\geq u.
 \end{align*}
Hence it is a uniform parabolic problem. The interior parabolic estimates gives the smoothness of $p_j$ in the interior and up to the top and bottom boundary away from the lateral face for $t\in[-1+\delta,0]$.

Now we show the convergence of $p_j\to u$ on the local domain
\begin{lemma}\label{lem 3.3}
	Set $D_5(t)=\left\{(x',x_n):x'\in B_5^{n-1},\ x_n\in\left(F_-,\widehat{f}(x',Z_*,t)\right)\right\}$. Then
	\begin{align*}
		d_j:=\sup_{D_5(t)\times[-1+\delta,0]}|p_j-u|\to 0.
	\end{align*}
\end{lemma}
\begin{proof}
	Consider a cylinder $\mathcal{Q}$ satisfying
	\begin{align*}
		D_{\text{loc}}(t)\times[-1+\delta,0]\Subset\mathcal{Q}=Q\times(-1,0],
	\end{align*}
	where $Q\Subset B_1$ is a fixed smooth domain. We claim for any $\eps>0$, there are strictly positive solutions $u_{\eps}$ satisfying equation \eqref{pme u} on $\mathcal{Q}$, with $u_{\eps}\geq \eps$, such that $u_{\eps}>u$ on $\overline{\mathcal{Q}}$. Also, when $\eps \to 0$, $u_{\eps}\to u$ uniformly and $u_{\eps}$ is smooth on any compact subset of $\mathcal{Q}$.
	
	Setting $u+\eps$ as the initial and lateral boundary data, we obtain a continuous solution $u_{\eps}$ from the standard parabolic theory. By comparison principle, $u_{\eps}\geq \eps$ and $u_{\eps}\to u$ uniformly. By interior parabolic estimates, $u_{\eps}$ is smooth on any compact subset of $\mathcal{Q}$. Then we need to verify the strict inequality $u_{\eps}>u$ on $\overline{\mathcal{Q}}$.
	
	Write $v_{\eps}=u_{\eps}^{m-1}$. On the set $\mathcal{Q}_{\eps}=\left\{u>\eps/2\right\}\cap \mathcal{Q}$, $v_{\eps}$ and $v$ are smooth. Then function $w=v_{\eps}-v\geq0$ satisfies
	\begin{align*}
		w_t=(m-1)v_{\eps}\Delta w+\left(\nabla v_{\eps}+\nabla v\right)\cdot \nabla w+(m-1)(\Delta v)w.
	\end{align*}
	The smoothness of $v$ on $\mathcal{Q}_{\eps}$ gives a constant $C_{\eps}>0$ such that $(m-1)\Delta v\geq -C_{\eps}$.
	
	On the initial and lateral boundary of $\mathcal{Q}$,
	\begin{align*}
		w=(u+\varepsilon)^{m-1}-u^{m-1}\geq \min_u\left((u+\varepsilon)^{m-1}-u^{m-1}\right)>0.
	\end{align*}
	On the remaining boundary $\left\{u=\eps/2\right\}$,
	\begin{align*}
		w\geq \eps^{m-1}-(\eps/2)^{m-1}>0.
	\end{align*}
	Choose a constant $\kappa_{\eps}>0$ satisfying
	\begin{align*}
		\kappa_{\eps}\leq \min_u\left((u+\varepsilon)^{m-1}-u^{m-1}\right),\qquad\kappa_{\eps}\leq \eps^{m-1}-(\eps/2)^{m-1}.
	\end{align*}
	Then it is easy to check function $\kappa_{\eps} e^{-C_{\eps}(t+1)}$ is the subsolution of the linear equation of $w$. The standard minimum argument, applied on $\mathcal{Q}_{\eps}$, gives
	\begin{align*}
		w\geq \kappa_{\eps} e^{-C_{\eps}(t+1)}>0\quad\text{on $\mathcal{Q}_{\eps}$}.
	\end{align*}
	Notice that a negative minimum of $w-\kappa_{\eps} e^{-C_{\eps}(t+1)}$ would occur in the interior of $\mathcal{Q}_{\eps}$, so no regularity of boundary $\left\{u=\eps/2\right\}$ is needed for minimum argument. When $u\leq \eps /2$, $u_{\eps}\geq \eps >u$. Therefore, we prove the claim.
	
	Fix $\eps>0$. Since $D_{\text{loc}}(t)\times[-1+\delta,0]\Subset\mathcal{Q}$, the claim shows that $u_{\eps}>u$ on $D_{\text{loc}}(t)\times[-1+\delta,0]$. The uniform convergence of $g_j\to u$ and $b_j\to u$ implies, for all sufficiently large $j$, depending on $\eps$,
	\begin{align*}
		g_j(x)\leq u_{\eps}(x,-1+\delta)\quad\text{on $D_{\text{loc}}(-1+\delta)$},\qquad b_j\leq u_{\eps}\quad\text{on $\partial D_{\text{loc}}(t)\times [-1+\delta,0]$}.
	\end{align*}
	Then the comparison principle gives $p_j\leq u_{\eps}$ on $D_{\text{loc}}(t)\times[-1+\delta,0]$. Since $p_j\geq u$, letting $j
	\to\infty$  first and then $\eps \to 0$, we finis the proof.
\end{proof}

Fourth, we construct $u_j$ and show $|u_j-p_j|\leq 1/j$. To achieve this, we introduce the penalty term in the equation of $u_j$. Now we define the penalty functions $l_j$ and $U_j$ using  suitable cut-off functions. Take a smooth function $\eta(x')\geq0$ satisfying
\begin{align*}
	\eta=\eta_0\quad\text{on $B_2^{n-1}$},\qquad \eta=0\quad\text{on $\left(B_3^{n-1}\right)^c$},
\end{align*}
where
\begin{align*}
	\eta_0>1+\sup_{(x',t)\in B_5\times[-1+\delta,0]}\left(\widehat{f}(x',Z_*,t)-F_-\right).
\end{align*}

Define on $\mathbb{T}^{n-1}\times \R \times[-1+\delta,0]$,
\begin{align*}
	l(x',x_n,t)=\widehat{u}(x',x_n-\eta(x'),t),\qquad U(x',x_n,t)=\widehat{u}(x',x_n+\eta(x'),t).
\end{align*}
The local cone condition gives $0\leq l\leq \widehat{u}\leq U\leq M_*$. Also the local cone condition holds for functions $l$ and $U$ with the constant $L_1= L_0+\|D\eta\|_{\infty}<\infty$. 

For $x'\in B_2^{n-1}$, $x_n\in\left(F_--1/2,\widehat{f}(x',Z_*,t)+1/2\right)$, $t\in[-1+\delta,0]$, one has
\begin{align*}
	l=0,\qquad U=M_*.
\end{align*}
For $x'\in\mathcal C=B_4^{n-1}\backslash B_3^{n-1}$, $x_n\in\left(F_-,\widehat{f}(x',Z_*,t)\right)$, $t\in[-1+\delta,0]$, one has
\begin{align*}
	l=U=u.
\end{align*}

Extend $l$ and $U$ in time by
\begin{align*}
	&l(x',x_n,t)=l(x',x_n,-1+\delta),&&\qquad U(x',x_n,t)=U(x',x_n,-1+\delta)\quad\text{for $t<-1+\delta$},\\
	&l(x',x_n,t)=l(x',x_n,0),&&\qquad U(x',x_n,t)=U(x',x_n,0)\quad\text{for $t>0$}.
\end{align*}
Modifying $l$ and $U$ with a nonnegative smooth space–time convolution kernel, with periodical convolution in $x'$, gives $l^0_j$ and $U^0_j$. They preserves the local cone condition with same constant $L_1$. The uniform convergence gives numbers $\omega_j\downarrow0$ such that
\begin{align*}
	|l^0_{j}-l|+|U^0_j-U|\leq \omega_j.
\end{align*}

Choose numbers $\sigma_j\downarrow 0$ and a smooth nondecreasing function $\beta_j$ satisfying
\begin{align*}
	\beta_j(q)=0\quad\text{for $q\leq0$},\qquad0\leq \beta_j(q)\leq (q)_+,\qquad (q)_+-\sigma_j\leq \beta_j(q).
\end{align*} 
Indeed, we can let $\beta_j=\int_{\R}S_j$, where smooth nondecreasing function $S_j$ satisfies
\begin{align*}
	S_j=0\quad\text{on $\left(-\infty,0\right)$},\qquad S_j=1\quad\text{on $(\sigma_j,\infty)$}.
\end{align*}
After $p_j$ and $d_j$ have been fixed, set
\begin{align*}
	B_j=2\left(\omega_j+d_j+\gamma_j+\sigma_j\right).
\end{align*}
For large $j$, we have $B_j\leq M_*$. 

Define
\begin{align*}
	l_j=\beta_j\left(l_j^0-B_j\right),\qquad U_j=M_*-\beta_j\left(M_*-U^0_j-B_j\right).
\end{align*}
They are smooth function and preserves the local cone condition with same constant $L_1$, since $\beta_j$ is nondecreasing. Moreover, on $\mathbb{T}^{n-1}\times \R\times\R$, we have
\begin{align*}
0\leq l_j\leq U_j\leq M_*,\qquad U_j\geq B_j.
\end{align*}
First inequality is because that for the nontrivial cases, $l_j= l^0_j-B_j\leq l_j^0-\omega_j\leq l$ and $U_j= U^0_j+B_j\geq U^0_j+\omega_j\geq U$. Hence $0\leq l_j\leq l\leq U\leq U_j\leq M_*$. Second inequality is because that for the nontrivial cases, $U_j= U^0_j+B_j\geq B_j$.

On $D_5(t)\times[-1+\delta,0]$, one has 
\begin{align*}
	\label{3.4}
	\tag{3.4}
	0\leq l_j\leq p_j\leq U_j\leq M_*.
\end{align*}
Indeed,
\begin{align*}
	l_j\leq \left(l+\omega_j+B_j\right)_+\leq(u-d_j)_+\leq p_j,
\end{align*}
where $l\leq \widehat{u}=u$ on $D_5(t)\times[-1+\delta,0]$. Similarly,
\begin{align*}
	U_j\geq \min\left\{M_*, U_j^0+B_j\right\}\geq \min\left\{M_*, u+d_j\right\}\geq p_j.
\end{align*}

For $x'\in \mathcal{C}=B^{n-1}_4\backslash B^{n-1}_3$, since $l=U=u$, computing gives
\begin{align*}
	\label{3.5}
	\tag{3.5}
	|l_j-p_j|+|U_j-p_j|\leq 2\left(d_j+\omega_j+B_j+\sigma_j\right)=:\eps_j\to0.
\end{align*}

For $x'\in B_1^{n-1},\ x_n\in\left(F_-,\widehat{f}(x',Z_*,t)\right)$, $t\in[-1+\delta,0]$, recall that $l=0$, $U=M_*$ on a bigger domain, then 
\begin{align*}
	l_j^0=l_j=0,\qquad U_j^0=U_j=M_*.
\end{align*}

After we define the penalty function and introduce the main properties, we begin to prove the smooth approximation property.
\begin{proof}[Proof of Proposition \ref{smooth appro}]
Define penalty term
\begin{align*}
	R_j(x,t,q)=j\beta_j(l_j-q)-j\beta_j(q-U_j).
\end{align*}
Recall that $D(t)=\left\{x'\in\mathbb{T}^{n-1},\ F_-<x_n<\widehat{f}\left(x',Z_*,t\right)\right\}$. Consider the following penalized problem for fixed $j$,
\begin{equation}
	\notag
	\left\{
	\begin{aligned}
		&(u_j)_t=\Delta(u_j^m)+R_j(x,t,u_j), &&(x,t)\in D(t)\times [-1+\delta,0],\\
		&u_j(x,-1+\delta)=g_j(x), &&x\in D(-1+\delta),\\
		&u_j(x,t)=\gamma_j, &&(x,t)\in \mathbb{T}^{n-1}\times\left\{x_n=F_-\right\}\times[-1+\delta,0],\\
		&u_j(x,t)=M_*, &&(x,t)\in \mathbb{T}^{n-1}\times\left\{x_n=\widehat{f}\left(x',Z_*,t\right)\right\}\times[-1+\delta,0].
	\end{aligned}
	\right.
\end{equation}
When $u_j=\gamma_j$, since $U_j\geq B_j\geq\gamma_j$, we have $R_j\geq 0$. When $u_j=M_*$, since $l_j\leq M_*$, we have $R_j\leq0$. Consequently the constants $\gamma_j$ and $M_*$ are a subsolution and a supersolution, respectively. Since $R_j$ is smooth, nonincreasing with respect to $u_j$, the classic parabolic theory gives a unique smooth solution $u_j$ satisfying
\begin{align*}
	\gamma_j\leq u_j\leq M_*.
\end{align*}

For fixed $u_j$, it is easy to check the function $R_j(x,t,u_j)$ preserves the local cone condition with constant $L_1$. Then, for any $e$ satisfying $e_n\geq L_1|e_{x'}|$, compare $u_j(x+\eps e,t)$ and $u_j(x,t)$ on the corresponding domain 
\begin{align*}
	D(t)\cap \left\{x'\in\mathbb{T}^{n-1},\ F_-<x_n+\eps e_n <\widehat{f}\left(x'+\eps e_{x'},Z_*,t\right)\right\}.
\end{align*}
Since $\left|D_{x'}\widehat{f}\right|\leq L_0\leq L_1$, we have 
\begin{align*}
	\widehat{f}\left(x'+\eps e_{x'},Z_*,t\right)-\eps e_n\leq \widehat{f}\left(x',Z_*,t\right).
\end{align*}
Therefore, above domain is $\left\{x'\in\mathbb{T}^{n-1},\ F_-<x_n <\widehat{f}\left(x'+\eps e_{x'},Z_*,t\right)-\eps e_n \right\}$. Notice that for this domain, we don't have the lateral boundary since $x'\in \mathbb{T}^{n-1}$. For the initial, top and bottom data, the order is trivial. Therefore, $u_j(x+\eps e)\geq u_j(x,t)$. Letting $\eps\to0$ we obtain the local cone condition holds for $u_j$ on $D(t)\times[-1+\delta,0]$ with same constant $L_1$, which is 
\begin{align*}
	D_{x_n} u_j\geq0,\qquad \left|D_{x'}u_j\right|\leq L_1 D_{x_n}u_j\quad\text{on $D(t)\times[-1+\delta,0]$}.
\end{align*}

Set $w=D_{x_n}u_j$ and the equation of $w$ is
\begin{align*}
	w_t=&mu_j^{m-1}\Delta w
	+2m(m-1)u_j^{m-2}Du_j\cdot Dw\\
	&+\left[m(m-1)u_j^{m-2}\Delta u_j
	+m(m-1)(m-2)u_j^{m-3}|Du_j|^2\right]w\\
	&-j\bigl[D\beta_j(l_j-u_j)+D\beta_j(u_j-U_j)\bigr]w\\
	&+jD\beta_j(l_j-u_j)D_{x_n}l_j
	+jD\beta_j(u_j-U_j)D_{x_n}U_j.
\end{align*}
The last two terms are nonnegative. Since $u_j$ is smooth and positive, the strong maximum principle for $w$ gives $w>0$ in the interior. Otherwise, we have $w\equiv0$, which implies $M=\gamma_j$, a contradiction. Since $D_{x_n}u_j>0$, the implicit function theorem supplies smooth function $f_j$.

For $x\in B_1^{n-1}$, $F_-\leq x_n\leq f(x',Z_*,t)$, $t\in[-1+\delta,0]$, since $l_j=0$, $U_j=M_*$, we obtain $R_j(x,t,u_j)=0$. Then the equation of $u_j$ turns into
\begin{align*}
	\left(u_j\right)_t=\Delta\left(u_j^m\right).
\end{align*}

It remains to show that $u_j\to u$ on above domain. First we show that $u_j\to p_j$ on $D_5(t)\times[-1+\delta,0]$. Set
\begin{align*}
	w_j=u_j-p_j,\qquad W_j=u_j^m-p_j^m.
\end{align*}
Subtracting their equations gives
\begin{align*}
 \left(w_j\right)_t=\Delta W_j+R_j(x,t,u_j).
\end{align*}
At the differentiable point
\begin{align*}
	\left(|w_j|\right)_t=\operatorname{sgn}\left(w_j\right)\left(w_j\right)_t=\operatorname{sgn}\left(w_j\right)\left(\Delta W_j+R_j(x,t,u_j)\right).
\end{align*}
Using the distribution inequality $\Delta|W_j|\geq \operatorname{sgn}(W_j)\Delta W_j=\operatorname{sgn}(w_j)\Delta W_j$ gives
\begin{align*}
	\label{3.6}
	\tag{3.6}
\left(|w_j|\right)_t-\Delta |W_j| \leq \operatorname{sgn}(w_j)R_j(x,t,u_j)\quad\text{in the distributional sense.}
\end{align*}
For $w_j=u_j-p_j>0$, if $u_j>U_j$, since \eqref{3.4} gives $l_j\leq p_j<u_j$, we have $R_j(x,t,u_j)< 0$. On the other hand, if $l_j\leq p_j<u_j\leq U_j$, then $R_j(x,t,u_j)=0$. For $w_j=u_j-p_j<0$, if $u_j<l_j$, since $u_j<p_j\leq U_j$, we have $R_j(x,t,u_j)>0$. On the other hand, if $l_j\leq u_j<p_j\leq U_j$, then $R_j(x,t,u_j)=0$. Consequently,
\begin{align*}
	\operatorname{sgn}(w_j)R_j(x,t,u_j)\leq 0.
\end{align*}

For $x'\in \mathcal C=B^{n-1}_4\backslash B^{n-1}_3$, \eqref{3.5} improves above inequality to
\begin{align*}
	\label{3.7}
	\tag{3.7}
		\operatorname{sgn}(w_j)R_j(x,t,u_j)\leq -j\left(|w_j|-\eps_j-\sigma_j\right)_+.
\end{align*}
In fact, for $w_j>0$, since \eqref{3.5} gives $U_j\leq p_j+\eps_j$, we have $u_j-U_j\geq u_j-\left(p_j+\eps_j\right)=w_j-\eps_j$. Then
\begin{align*}
	\beta_j(u_j-U_j)\geq\beta_j(w_j-\eps_j)\geq \left(w_j-\eps_j-\sigma_j\right)_+.
\end{align*}
On the other hand, since $u_j>p_j\geq l_j$, we have $\beta_j(l_j-u_j)=0$. Therefore
\begin{align*}
	\operatorname{sgn}(w_j)R_j(x,t,u_j)\leq -j\left(|w_j|-\eps_j-\sigma_j\right)_+.
\end{align*}
For $w_j<0$, similarly
\begin{align*}
	&\beta_j(l_j-u_j)\geq \beta_j(|w_j|-\eps_j)\geq \left(|w_j|-\eps_j-\sigma_j\right)_+,\\
	&\beta_j(u_j-U_j)=0.
\end{align*}
Therefore we obtain \eqref{3.7}.

Choose a nonnegative smooth cutoff $\xi(x')$ satisfying
\begin{align*}
	\xi=1\quad\text{on $B_4^{n-1}$},\qquad \operatorname{supp}\xi\Subset B_5^{n-1}.
\end{align*}
Multiplying \eqref{3.6} by $\xi$ and integrating on $D_5(t)\times[-1+\delta,0]$ gives
\begin{align*}
	\int_{-1+\delta}^0\int_{D_5(t)} \xi\left(|w_j|\right)_t -\xi\operatorname{sgn}(w_j)R_j(x,t,u_j)dxdt\leq \int_{-1+\delta}^0\int_{D_5(t)}\xi\Delta|W_j|dxdt.
\end{align*}
For the right term, integration by parts gives
\begin{align*}
	\int_{D_5(t)} \xi\Delta|W_j|=\int_{D_5(t)} |W_j|\Delta\xi+\int_{\partial D_5(t)}\left(\xi D_{\nu}|W_j|-|W_j|D_{\nu}\xi\right),
\end{align*}
where $\nu$ is the outer normal vector. For $x'\in \partial B_5^{n-1}$, $\xi=D_{\nu}\xi=0$, hence we only need to consider the top and bottom boundary. Since $u_j=p_j$ on the top and bottom boundary, $W_j=0$. Also, $D_{\nu}|W_j|=-|D_{\nu}W_j|\leq 0$ on the top and bottom boundary. Therefore
\begin{align*}
	\int_{D_5(t)} \xi\Delta|W_j|\leq \int_{D_5(t)} |W_j|\Delta\xi.
\end{align*}
For the first term, we know
\begin{align*}
	\frac{d}{dt}\int_{D_5(t)}\xi|w_j|=\int_{D_5(t)}\xi\left(|w_j|\right)_t+\int_{\partial D_5(t)}\xi|w_j|\left(\nu_B\cdot\nu\right),
\end{align*}
where $\nu_B$ is the velocity vector of $\partial D_5(t)$. Since $w_j=0$ on the top and bottom boundary, and $\xi=0$ for $x'\in \partial B_5^{n-1}$
\begin{align*}
	\int_{-1+\delta}^0\int_{D_5(t)} \xi\left(|w_j|\right)_t=\int_{D_5(0)} \xi|w_j|(x,0)-\int_{D_5(-1+\delta)} \xi|w_j|(x,-1+\delta)=\int_{D_5(0)} \xi|w_j|(x,0)\geq0.
\end{align*}
Therefore
\begin{align*}
	\int_{-1+\delta}^0\int_{D_5(t)} -\xi\operatorname{sgn}(w_j)R_j(x,t,u_j)\leq \int_{-1+\delta}^0\int_{D_5(t)}|W_j|\Delta \xi \leq C_0,
\end{align*}
where the constant $C_0>0$ is independent of $j$, because $|W_j|\leq 2M_*^m$. Moreover, since $-\operatorname{sgn}(w_j)R_j(x,t,u_j)\geq 0$, restricting $x'\in \mathcal C$ and \eqref{3.7} gives
\begin{align*}
	\int_{-1+\delta}^0\int_{\mathcal C\times(F_-,f(x',Z_*,t))} j\left(|w_j|-\eps_j-\sigma_j\right)_+\leq C_0.
\end{align*}
Consequently
\begin{align*}
	\label{3.8}
	\tag{3.8}
	\int_{-1+\delta}^0\int_{\mathcal C\times(F_-,f(x',Z_*,t))} |w_j|\leq C_0\left(j^{-1}+\eps_j+\sigma_j\right),
\end{align*}
for a new constant $C_0>0$.

Choose a nonnegative smooth cutoff $\chi(x')$ satisfying
\begin{align*}
 \chi=1\quad\text{on $B_3^{n-1}$},\qquad\operatorname{supp} \chi \Subset B_4^{n-1},\qquad \operatorname{supp} D\chi \Subset \mathcal C.
\end{align*}
Multiplying $\left(|w_j|\right)_t\leq \Delta|W_j|$ by $\chi$ and integrating on $D_5(t)\times[-1+\delta,s]$ for any $s\in[-1+\delta,0]$ gives
\begin{align*}
	\int_{D_5(s)}\chi|w_j|(x,s)\leq \int_{-1+\delta}^0\int_{D_5(t)}|W_j|\Delta \chi\leq mM_*\int_{-1+\delta}^0\int_{\mathcal C\times(F_-,f(x',Z_*,t))} |w_j||\Delta\chi|.
\end{align*}
Using \eqref{3.8} gives
\begin{align*}
	\label{3.9}
	\tag{3.9}
	\sup_{-1+\delta\leq s\leq 0}\int_{D_5(s)}\chi|w_j|(x,s)\leq C_0\left(j^{-1}+\eps_j+\sigma_j\right),
\end{align*}
for a new constant $C_0>0$.

Lemma \ref{lem 3.3} and \eqref{3.9} give $u_j\to u$ in $L^1$ uniformly in $t\in\left[-1+\delta,0\right]$ , for $x'\in B_3^{n-1}$. Moreover, for $x'\in B_1^{n-1}$, equation \eqref{pme u} shows the uniformly bounded $u_j$ is locally H\"older continuous with uniform H\"older norm. Since $u$ is also H\"older continuous, we can improved the convergence in $L^1$ into the uniform convergence for $x'\in B_*^{n-1}$. On the compact of $\ome$, the interior parabolic estimates give the convergences with all spatial and time derivatives. For fixed $0<Z_0<Z_*$, the uniform bounds of all spatial and time derivatives of $f_j(x',Z_0,t)$ follow from above convergences.

\end{proof}



\section{Upper bound gradient estimates}\label{s4}
In this section, we establish the upper gradient estimate for smooth solution. First, we prove the central estimate for level set function $f$. Choose
\begin{align*}
	D_1(t)\Subset\left\{x'\in B_*^{n-1},\ x_n\in\left(F_-,f(x',Z_*,t)\right)\right\},\qquad \mathcal{D}=D_1(t)\times[-1+2\delta,0].
\end{align*}
where $B_{*}^{n-1}$, $F_-$, $Z_*$ and $\delta$ is given in Proposition \ref{smooth appro}.
\begin{theorem}\label{thm up grad}
	Let $v$ be a smooth positive solution of \eqref{pme v} on $\mathcal{D}$ with
	\begin{align*}
		v\in[\gamma,Z_*],\qquad D_{x_n}v>0,\qquad |D_{x'}v|\leq L D_{x_n}v,\qquad \Delta v\geq -K_\delta,
	\end{align*}
	for an arbitrarily small constant $\gamma>0$, a constant $L\in[0,\infty)$ and the constant $K_{\delta}$ is given in Proposition \ref{local AB prop}. Suppose that the level set function $f(x',z,t)$ satisfies
	\begin{align*}
		F_-\leq f\leq F_+\quad\text{on $\mathcal{D}$},
	\end{align*}
	for a constant $F_+>F_-$. Let $\chi=\chi(x')\in[0,1]$ be a smooth cut-off function satisfying
	\begin{align*}
		\operatorname{supp}\chi\Subset \pi_{x'}\mathcal D,\qquad |D\chi|^2\leq C_{\chi} \chi,\qquad|D^2\chi|\leq C_{\chi}\quad\text{for a constant $C_{\chi}>0$.}
	\end{align*}
	 Let $\vartheta=\vartheta(x_n)\in\left[0,1\right]$ be a smooth nondecreasing cut-off function satisfying
	\begin{align*}
		&\operatorname{supp}\vartheta \Subset \left(F_-,F_+\right],\qquad\vartheta(F_-+\eps_0)=1\quad\text{for a fixed small $\eps_0$},\\
		&|D\vartheta|^2\leq C_{\vartheta}\vartheta,\qquad|D^2\vartheta|\leq C_{\vartheta}\quad\text{for a constant $C_{\vartheta}>0$.}
	\end{align*}

	Then there exists constants $Z_1\in(0,Z_*)$ and $C_1>0$, depending only on $n,m,L,F_+-F_-,K_{\delta},C_{\chi},C_{\vartheta}$ and independent of $\gamma$, such that for any $\gamma<Z\leq Z_1$, if at $z=Z$, $\chi\vartheta>0$, $-1+2\delta<t\leq0$, there exists a constant $H_1$ such that
	\begin{align*}
		\frac{-f_t}{H}\leq H_1,
	\end{align*}
	where
	\begin{align*}
		\label{def H}
		\tag{4.1}
		H=1+F_+-f+\frac{L^2-|D_{x'}f|^2}{2}+z\sqrt{f_z}\left(1+9L^2-|D_{x'}f|^2\right),
	\end{align*}
	then for $\gamma<z\leq Z$, $\chi=\vartheta=1$, $-1+2\delta<t\leq 0$, we have
	\begin{align*}
		\frac{-f_t}{H}\leq C_1\left(1+H_1+(t+1-2\delta)^{-3/5}\right).
	\end{align*}
\end{theorem}

We demonstrate the main steps first. Let $\tau=t+1-\delta$ and consider function
\begin{align*}
	\mathcal{P}=\tau e^z\chi(x')\vartheta(f)\frac{-f_t}{H}.
\end{align*}
We will use the maximum argument for $\log \mathcal{P}$ on $\mathcal{D}$ with respect to a linearization operator $\mathcal{L}$. By our assumption, $\log \mathcal P$ is controlled on $\partial \mathcal D$. Therefore, we only consider the interior maximum point. Then, using the equation of $f$ and the local cone condition we prove that at the interior maximum point, $\log \mathcal P$ can be controlled. Before we begin to prove, we make necessary computations to define desired linearization operator $\mathcal{L}$.

Implicit differentiation gives
\begin{align*}
	&D_{x_n}v=\frac{1}{f_z},\qquad D_{x'}v=-\frac{D_{x'}f}{f_z},\qquad v_t=-\frac{f_t}{f_z},\\
	&\Delta v=-\frac{\Delta_{x'}f}{f_z}+\frac{2D_{x'}f\cdot D_{x'}f_z}{f_z^2} -\frac{1+|D_{x'}f|^2}{f_z^3}f_{zz}.
\end{align*}
Set 
\begin{align*}
	F=\frac{1+|D_{x'}f|^2}{f_z}
\end{align*}
then
\begin{align*}
	F_z=\frac{2D_{x'}f \cdot D_{x'}f_z}{f_z}-\frac{1+|D_{x'}f|^2}{f_z^2}f_{zz}.
\end{align*}
Hence
\begin{align*}
	\Delta v=-\frac{\Delta_{x'}f-F_z}{f_z}.
\end{align*}
Substitution into equation \eqref{pme v} gives
\begin{align*}
	\label{pme f}
	\tag{4.2}
	f_t=(m-1)z\left(\Delta_{x'}f-F_z\right)-F.
\end{align*}
The divergence form is 
\begin{align*}
	f_t=(m-1)z^{\frac{m-2}{m-1}} \operatorname{div}\left(z^{\frac{1}{m-1}}\left(D_{x'}f,-F\right)\right)
\end{align*}

Define the linearization operator $T$ as
\begin{align*}
	Tg=g_t-(m-1)z^{\frac{m-2}{m-1}} \operatorname{div}\left(z^{\frac{1}{m-1}}\mathcal{A}Dg\right),
\end{align*}
where
\begin{align*}
	\mathcal A=
	\begin{pmatrix} 
		I_{n-1}&0\\
		\left(\frac{-2D_{x'}f}{f_z}\right)^T&\frac{1+|D_{x'}f|^2}{f_z^2}
	\end{pmatrix}
\end{align*}
is the derivative of $\left(D_{x'}f,-F\right)$ with respect to $Df=(D_{x'}f,f_z)$. . Its quadratic form is
\begin{align*}
	\left(\xi,\eta\right)\mathcal{A}(\xi,\eta)=\left|\xi-f_z^{-1}D_{x'}f\cdot\eta\right|^2+f_z^{-2}\eta^2>0\quad\text{for $(\xi,\eta)\neq0$}.
\end{align*}
Therefore, $\mathcal A$ is positive definite. Differentiation \eqref{pme f} with respect to $x_i$, $z$ and $t$ gives
\begin{align*}
	\label{TDf}
	\tag{4.3}
	TD_{x_i}f=0,\qquad Tf_t=0,\qquad Tf_z=(m-1)\left(\Delta_{x'}f-F_z\right)=\frac{F+f_t}{z}.
\end{align*}

For scalar functions $g$ and $h$, denote
\begin{align*}
	\label{def lam}
	\tag{4.4}
	\Lambda\left(g,h\right)=(m-1)z\left[\left(D_{x'}g-\frac{D_{x'}f}{f_z}g_z\right)\cdot\left(D_{x'}h-\frac{D_{x'}f}{f_z}h_z\right)+\frac{g_zh_z}{f_z^2}\right].
\end{align*}
Then the chain rules gives for smooth function $\Phi(g_1,\cdots,g_i)$
\begin{align*}
	\label{product}
	\tag{4.5}
	\begin{split}
		T(gh)&=gTh+hTg-2\Lambda(g,h),\\
		T\Phi(g_1,\cdots,g_i)&=\sum_{j=1}^{i} D_{g_j}\Phi Tg_j-\sum_{j,k=1}^{i}D^2_{g_jg_k}\Phi \Lambda(g_j,g_k).
	\end{split}
\end{align*}

Denote for $i,j=1,\cdots,n-1$
\begin{align*}
	M_{ij}&=D^2_{x_ix_j}f-\frac{D_{x_i}f_z D_{x_j}f}{f_z},\qquad E=|M|^2+\left|\frac{D_{x'}f_z}{f_z}\right|^2,\qquad \phi=\frac{D_{x'}f_z}{f_z}-\frac{D_{x'}f }{f_z^2}f_{zz},\\
	Q&=|\phi|^2+\left|\frac{f_{zz}}{f_z^2}\right|^2,\qquad S=\operatorname{tr} M-\frac{D_{x'}f\cdot D_{x'}f_z}{f_z}+\frac{1+|D_{x'}f|^2}{f_z^2}f_{zz}.
\end{align*}
Then the definition of $S$ and equation \eqref{pme f} gives
\begin{align*}
	\label{use S}
	\tag{4.6}
	f_t+F=(m-1)zS.
\end{align*}
Combining with the notations and \eqref{use S}, the computations give
\begin{align*}
	\label{computation}
	\tag{4.7}
	\begin{split}
		&Tf=2(m-1)z\frac{f_{zz}}{f_z^2}-\frac{2}{f_z},\\ &Tz=\frac{F+2f_t}{f_z}-2(m-1)z\frac{\operatorname{tr} M}{f_z},\\
		&\sum_{i=1}^{n-1}\Lambda(D_{x_i}f,D_{x_i}f)=(m-1)zE,\\
		&\Lambda(f_z,f_z)=(m-1)zf_z^2Q,\\
		&\Lambda(D_{x_i}f,f_z)=(m-1)zf_z\left[\left(M\phi\right)_i+\frac{D_{x_i}f_z}{f_z^3}f_{zz}\right],\\
		&\Lambda(z,D_{x_i}f)=(m-1)z\frac{D_{x_i}f_z/f_z-\left(MD_{x'}f\right)_i}{f_z},\\
		&\Lambda(z,f_z)=(m-1)z\left(\frac{f_{zz}}{f_z^2}-D_{x'}f\cdot\phi\right).
	\end{split}
\end{align*}

First, we estimate $TH$. From \eqref{def H}, we know
\begin{align*}
	\label{bdd H}
	\tag{4.8}
	1+z\sqrt{f_z}\leq H\leq C_H\left(1+z\sqrt{f_z}\right),
\end{align*}
where
\begin{align*}
	C_H=\max\left\{1+F_+-F_-+L^2/2, 1+9L^2\right\}.
\end{align*}
Define functions $A=A(D_{x'}f)$ and $G=G(D_{x'}f,f_z)$ as 
\begin{align*}
A(D_{x'}f)=1+9L^2-|D_{x'}f|^2,\qquad G(D_{x'}f,f_z)=\sqrt{f_z}A.
\end{align*}
By \eqref{TDf}, \eqref{product} and \eqref{computation}
\begin{align*}
	TG=\frac{A}{2\sqrt{f_z}}\frac{F+f_t}{z}+\mathcal{G},
\end{align*}
where
\begin{align*}
	\mathcal{G}=(m-1)z\sqrt{f_z}\left[2E+2D_{x'}f\cdot\left(M\phi+\frac{D_{x'}f_z}{f_z^3}f_{zz}\right)+\frac{A}{4}Q\right].
\end{align*}
Since $|D_{x'}f|=|D_{x'}v|/D_{x_n}v\leq L$,
\begin{align*}
	&2E+2D_{x'}f\cdot\left(M\phi+\frac{D_{x'}f_z}{f_z^3}f_{zz}\right)+\frac{A}{4}Q\\
	\geq&2E-\left(|M|^2+L^2|\phi|^2\right)-\left(\left|\frac{D_{x'}f_z}{f_z}\right|^2+L^2\left|\frac{f_{zz}}{f_z^2}\right|^2\right)+\frac{1+8L^2}{4}Q\\
	=&E+\left(\frac14+L^2\right)Q\geq c_1\left(E+Q\right),
\end{align*}
where $c_1=\min\left\{1, 1/4+L^2\right\}>0$. Hence
\begin{align*}
	\mathcal{G}\geq (m-1)c_1z\sqrt{f_z}\left(E+Q\right).
\end{align*}
Also,
\begin{align*}
	2\Lambda(z,G)=(m-1)zf_z^{-\frac12}\left[4D_{x'}f\cdot \left(MD_{x'}f-\frac{D_{x'}f_z}{f_z}\right)+A\left(\frac{f_{zz}}{f_z}-D_{x'}f\cdot\phi\right)\right]
\end{align*}
By above identity, \eqref{product}, \eqref{use S} and \eqref{computation}
\begin{align*}
	T\left(zG\right)=-\frac{1+\left|D_{x'}f\right|^2}{f_z^{3/2}}A+z\mathcal{G}+(m-1)zf_z^{-\frac12}\mathcal{R},
\end{align*}
where
\begin{align*}
	\mathcal{R}=\frac{A}{2}\operatorname{tr}M+\frac{3A}{2}\left(\frac{f_{zz}}{f_z^2}-D_{x'}f\cdot\phi\right)+4D_{x'}f\cdot \phi+\frac{4\left|D_{x'}f\right|^2}{f_z^2}f_{zz}-4D_{x'}f\cdot MD_{x'}f.
\end{align*}
Since $\left|\operatorname{tr}M\right|\leq \sqrt{n-1}|M|$ and $\dis \left|\frac{f_{zz}}{f_z^2}-D_{x'}f\cdot\phi\right|\leq \sqrt{1+\left|D_{x'}f\right|^2}\sqrt{Q}\leq \sqrt{1+L^2}\sqrt{Q}$,
\begin{align*}
	|\mathcal{R}|&\leq \frac{1+9L^2}{2}\sqrt{n-1}|M|+\frac{3(1+9L^2)}{2}\sqrt{1+L^2}\sqrt{Q}+4L|\phi|+4L^2\left|\frac{f_{zz}}{f_z^2}\right|+4L^2|M|\\
	&=\left(\frac{1+9L^2}{2}\sqrt{n-1}+4L^2\right)|M|+\left(\frac{3(1+9L^2)}{2}\sqrt{1+L^2}+4L+4L^2\right)\sqrt{Q}\\
	&\leq \left(\frac{1+9L^2}{2}\sqrt{n-1}+4L^2\right)\sqrt{E}+\left(\frac{3(1+9L^2)}{2}\sqrt{1+L^2}+4L+4L^2\right)\sqrt{Q}\\
	&\leq C_{\mathcal R}\sqrt{E+Q},
\end{align*}
where
\begin{align*}
	C_{\mathcal{R}}=\sqrt{\left(\frac{1+9L^2}{2}\sqrt{n-1}+4L^2\right)^2+\left(\frac{3(1+9L^2)}{2}\sqrt{1+L^2}+4L+4L^2\right)^2}.
\end{align*}
Young's inequality gives
\begin{align*}
	(m-1)C_{\mathcal{R}}zf_z^{-\frac12}\sqrt{E+Q}\leq \frac{(m-1)c_1}{4}z^2f_z^{\frac{1}{2}}(E+Q)+\frac{(m-1)C_{\mathcal{R}}^2}{c_1}f_z^{-\frac32}.
\end{align*}
Consequently,
\begin{align*}
	\label{bdd TzG1}
	\tag{4.9}
	T\left(zG\right)\geq -\left[(1+L^2)(1+9L^2)+\frac{(m-1)C_{\mathcal{R}}^2}{c_1}\right]f_z^{-\frac32}+\frac{3(m-1)c_1}{4}z^2f_z^{\frac12}(E+Q).
\end{align*}

On the other hand, \eqref{use S} gives
\begin{align*}
	f_t^2=\left[F-(m-1)zS\right]^2\leq 2F^2+2(m-1)^2z^2S^2.
\end{align*}
Since $\dis S=\operatorname{tr} M+\frac{f_{zz}}{f_z^2}-D_{x'}f\cdot\phi$,
\begin{align*}
	|S|&\leq \left|\operatorname{tr}M\right|+\left|\frac{f_{zz}}{f_z^2}-D_{x'}f\cdot\phi\right|\leq \sqrt{n-1}\sqrt{E}+\sqrt{1+L^2}\sqrt{Q}\\
	&\leq\sqrt{n+L^2}\sqrt{E+Q}.
\end{align*}
Recall that $\dis F=\frac{1+|D_{x'}f|^2}{f_z}$, then
\begin{align*}
	f_t^2\leq 2\frac{(1+L^2)^2}{f_z^2}+2(m-1)^2z^2(n+L^2)(E+Q).
\end{align*}
Multiply above inequality by $c_T f_z^{\frac12}$, where $\dis c_T=\frac{c_1}{8(m-1)(n+L^2)}$
\begin{align*}
	c_T f^{\frac12}_z f_t^2\leq \frac{c_1(1+L^2)^2}{4(m-1)(n+L^2)}f_z^{-\frac32}+\frac{(m-1)c_1}{4}z^2f_z^{\frac12}(E+Q).
\end{align*}
Substitute into \eqref{bdd TzG1}
\begin{align*}
	T\left(zG\right)\geq \frac{(m-1)c_1}{2}z^2f_z^{\frac12}(E+Q)+c_Tf_z^{\frac12}f_t^2-Cf_z^{-\frac32},
\end{align*}
for a constant $C=C(m,n,L)$.

Recall that $\dis H=1+F_+-f+\frac{L^2-|D_{x'}f|^2}{2}+zG$, then \eqref{product} and \eqref{computation} gives
\begin{align*}
	TH=-Tf+(m-1)zE+T\left(zG\right).
\end{align*}	
Using \eqref{use S} we also have
\begin{align*}
	Tf=\frac{2f_t}{1+\left|D_{x'}f\right|^2}-\frac{2(m-1)z}{1+\left|D_{x'}f\right|^2}\left(\operatorname{tr}M-\frac{D_{x'}f\cdot D_{x'}f_z}{f_z}\right).
\end{align*}
Notice that
\begin{align*}
	\label{ineq trM}
	\tag{4.10}
	\left|\operatorname{tr}M-\frac{D_{x'}f\cdot D_{x'}f_z}{f_z}\right|\leq \sqrt{n-1}|M|+L\left|\frac{D_{x'}f_z}{f_z}\right|\leq \sqrt{n-1+L^2}\sqrt{E}.
\end{align*}
Hence
\begin{align*}
	-Tf+(m-1)zE&\geq \frac{-2f_t}{1+|D_{x'}f|^2}+(m-1)z\left(E-2\sqrt{n-1+L^2}\sqrt{E}\right)\\
	&\geq\frac{-2f_t}{1+|D_{x'}f|^2}+(m-1)z\left(\frac{E}{2}-2(n-1+L^2)\right).
\end{align*}
Consequently,
\begin{align*}
	TH\geq& c_E (m-1)z\left(1+zf_z^{\frac12}\right)E+\frac{c_1}{2}(m-1)z^2f_z^{\frac12}Q\\
	&-\frac{2f_t}{1+|D_{x'}f|^2}+c_T f_z^{\frac12}f_t^2-C(m-1)z-Cf_z^{-\frac32},
\end{align*}
for a new constant $C=C(m,n,L)$, where $c_E=1/2\min\left\{1,c_1\right\}$.

Set
\begin{align*}
	B=\frac{z\sqrt{f_z}}{H}\in\left(0,1\right).
\end{align*}
The range uses \eqref{bdd H}, which also gives
\begin{align*}
	\label{esti TH}
	\tag{4.11}
	\begin{split}
	\frac{TH}{H}\geq&\frac{c_E}{C_H}(m-1)zE+\frac{c_1}{2}(m-1)zBQ+c_T\sqrt{f_z}\frac{f_t^2}{H}\\
	&-\frac{2}{1+\left|D_{x'}f\right|^2}\frac{f_t}{H}-C\frac{(m-1)z+f_z^{-\frac32}}{H}\\
	\geq&c_*\left[(m-1)zE+(m-1)zBQ+\sqrt{f_z}\frac{f_t^2}{H}\right] \\ &-\frac{2}{1+\left|D_{x'}f\right|^2}\frac{f_t}{H}-C\frac{(m-1)z+f_z^{-\frac32}}{H},
	\end{split}
\end{align*}
where $\dis c_*=\min\left\{\frac{c_E}{C_H},\frac{c_1}{2},c_T\right\}$.

Second, we define the linearization operator $\mathcal L$ and obtain the estimates for $\mathcal{L}f$, $\mathcal{L}\log\chi$, $\mathcal{L}z$ and $\Lambda\left(\log\vartheta+\log\chi+z,\log\vartheta+\log\chi+z\right)$. 

Define
\begin{align*}
	\mathcal Lg=Tg-\frac{2\Lambda\left(H,g\right)}{H},\qquad w=\frac{-f_t}{H}.
\end{align*}
Since $Tf_t=0$, by \eqref{product},
\begin{align*}
	0=T\left(wH\right)=wTH+HTw-2\Lambda(H,w).
\end{align*}
Then
\begin{align*}
	\label{Lw}
	\tag{4.12}
	\mathcal Lw=-w\frac{TH}{H}.
\end{align*}

Differentiation of $H$ gives
\begin{align*}
	D_{x'}H-\frac{D_{x'}f}{f_z}H_z&=-\left(1+2z\sqrt{f_z}\right)M^TD_{x'}f+\frac{zG\phi}{2}-\frac{AD_{x'}f}{\sqrt{f_z}}\\
	\frac{H_z}{f_z}&=-1-\left(1+2z\sqrt{f_z}\right)\frac{D_{x'}f\cdot	D_{x'}f_z}{f_z}+\frac{A}{\sqrt{f_z}}+\frac{zGf_{zz}}{2f_z^2}.
\end{align*}
Then from \eqref{def lam} and computations we obtain for $\chi>0$
\begingroup
\allowdisplaybreaks[1]
\begin{align}
	\label{computation2}
	\tag{4.13}
	\Lambda\left(f,f\right)=&(m-1)z\notag\\
	\Lambda\left(f,\log \chi\right)=&0\notag\\
	\Lambda\left(f,z\right)=&(m-1)zf_z^{-1}\notag\\
	\Lambda\left(H,f\right)=&(m-1)z\left[-1-\left(1+2z\sqrt{f_z}\right)\frac{D_{x'}f\cdot D_{x'}f_z}{f_z}+\frac{A}{\sqrt{f_z}}+\frac{zGf_{zz}}{2f_z^2}\right]\notag\\
	\Lambda\left(H,\log \chi\right)=&(m-1)z D\left(\log \chi\right)\cdot\left[-\left(1+2z\sqrt{f_z}\right)M^TD_{x'}f+\frac{zG\phi}{2}-\frac{AD_{x'}f}{\sqrt{f_z}}\right]\notag\\
	\Lambda(H,z)=&(m-1)zf_z^{-1}\left[-1-\frac{D_{x'}f\cdot D_{x'}f_z}{f_z}+D_{x'}f\cdot MD_{x'}f\right]\notag\\*
	&+(m-1)z^2f_z^{-\frac12}\left[\frac{A}{2}\left(\frac{f_{zz}}{f_z^2}-D_{x'}f\cdot\phi\right)-2D_{x'}f\cdot\left(\frac{D_{x'}f_z}{f_z}-MD_{x'}f\right)\right]\notag\\*
	&+(m-1)zA\frac{1+\left|D_{x'}f\right|^2}{f_z^{\frac32}}\notag\\
	T\left(\log \chi\right)=&-(m-1)z\Delta \left(\log \chi\right)+\frac{-2f_t}{1+\left|D_{x'}f\right|^2}D_{x'}f\cdot D\left(\log\chi\right)\notag\\*
	&+2(m-1)z D\left(\log\chi\right)\cdot\left[\frac{D_{x'}f_z}{f_z}+\frac{\operatorname{tr}M-f_z^{-1}D_{x'}f\cdot D_{x'}f_z }{1+\left|D_{x'}f\right|^2}D_{x'}f\right].\notag
\end{align}
\endgroup

Then we start to estimate $\mathcal{L}f$, $\mathcal{L}\log\chi$, $\mathcal{L}z$ and $\Lambda\left(\log\vartheta+\log\chi+z,\log\vartheta+\log\chi+z\right)$ with \eqref{computation} and \eqref{computation2}.  

Set
\begin{align*}
	\Theta=\frac{zG}{2H}\in\left(0,\frac12\right).
\end{align*}

For the first term $\mathcal{L}f$, if $f_t<0$, using \eqref{ineq trM} we obtain
\begin{align*}
	\mathcal{L}f=&2\frac{1-\Theta}{1+\left|D_{x'}f\right|^2}f_t-2\Theta f_z^{-1}-2(m-1)z\frac{1-\Theta}{1+\left|D_{x'}f\right|^2}\left(\operatorname{tr}M-\frac{D_{x'}f\cdot D_{x'}f_z}{f_z}\right)\\
	&+\frac{2(m-1)z}{H}\left[1+\left(1+2z\sqrt{f_z}\right)\frac{D_{x'}f\cdot D_{x'}f_z}{f_z}-\frac{A}{\sqrt{f_z}}\right]\\
	\leq&C(m-1)z\sqrt{E}+2(m-1)z+4L(m-1)z\sqrt{E}\\
	\leq& C(m-1)z\sqrt{E}+2(m-1)z,
\end{align*}
for a new constant $C=C(m,n,L)$. Multiplying above inequality by $D\left(\log\vartheta\right)>0$ and using Young's inequality
\begin{align*}
	C(m-1)z D\left(\log\vartheta\right)\sqrt{E}\leq \eps (m-1)zE+C_{\eps}(m-1)z\left(D\left(\log\vartheta\right)\right)^2,
\end{align*}
we obtain 
\begin{align*}
	\label{esti Lf}
	\tag{4.14}
	D\left(\log\vartheta\right)\mathcal{L}f\leq\varepsilon(m-1)zE+C(m-1)z\left[\left(D\left(\log\vartheta\right)\right)^2+D\left(\log\vartheta\right)\right],
\end{align*}
for a new constant $C=C(m,n,L,\varepsilon)$, whenever $f_t<0$. In the sequel till the end of this section, $C=C(m,n,L,\eps)$ denotes different constants.

For the second term $\mathcal{L}\log\chi$, use\eqref{ineq trM}
\begin{align*}
	\mathcal{L}\left(\log\chi\right)=&-(m-1)z\Delta \left(\log \chi\right)+\frac{-2f_t}{1+\left|D_{x'}f\right|^2}D_{x'}f\cdot D\left(\log\chi\right) \\
	&+2(m-1)z D\left(\log\chi\right)\cdot\left[\frac{D_{x'}f_z}{f_z}+\frac{\operatorname{tr}M-f_z^{-1}D_{x'}f\cdot D_{x'}f_z }{1+\left|D_{x'}f\right|^2}D_{x'}f\right]\\
	&-\frac{2(m-1)z}{H}D\left(\log \chi\right)\cdot\left[-\left(1+2z\sqrt{f_z}\right)M^TD_{x'}f+\frac{zG\phi}{2}-\frac{AD_{x'}f}{\sqrt{f_z}}\right]\\
	\leq&(m-1)z\sqrt{N}|D^2\left(\log\chi\right)|+2L|f_t|\left|D\left(\log\chi\right)\right|+C(m-1)z\sqrt{E}\left|D\left(\log\chi\right)\right|\\
	&+C(m-1)zB\sqrt{Q}\left|D\left(\log\chi\right)\right|+C(m-1)f_z^{-1}\left|D\left(\log\chi\right)\right|.
\end{align*}
Here, the last term use
\begin{align*}
	\frac{z}{\sqrt{f_z}H}=f_z^{-1}B\leq f_z^{-1}.
\end{align*}
Young's inequality gives
\begin{align*}
	C(m-1)z\sqrt{E}\left|D\left(\log\chi\right)\right|\leq& \eps(m-1)zE+C_{\eps}(m-1)z\left|D\left(\log\chi\right)\right|^2,\\
	C(m-1)zB\sqrt{Q}\left|D\left(\log\chi\right)\right|\leq&\eps (m-1)zBQ+C_{\eps}(m-1)z\left|D\left(\log\chi\right)\right|^2.
\end{align*}
Thus
\begin{align*}
	\label{esti Lchi}
	\tag{4.15}
	\begin{split}
	\mathcal{L}\left(\log\chi\right)\leq&\eps(m-1)z\left(E+BQ\right)+C\left(|f_t|+f_z^{-1}\right)\left|D\left(\log\chi\right)\right|\\&+C(m-1)z\left(\left|D\left(\log\chi\right)\right|^2+\left|D^2\left(\log\chi\right)\right|\right)
	\end{split}
\end{align*}

For the third term $\mathcal{L}z$, 
\begin{align*}
	\mathcal{L}z=&\frac{F+2f_t}{f_z}-2(m-1)z\frac{\operatorname{tr}M}{f_z}-\frac{2(m-1)z}{H}f_z^{-1}\left[-1-\frac{D_{x'}f\cdot D_{x'}f_z}{f_z}+D_{x'}f\cdot MD_{x'}f\right]\\
	&-\frac{2(m-1)z^2}{H}f_z^{-\frac12}\left[\frac{A}{2}\left(\frac{f_{zz}}{f_z^2}-D_{x'}f\cdot\phi\right)-2D_{x'}f\cdot\left(\frac{D_{x'}f_z}{f_z}-MD_{x'}f\right)\right]\\
	&-\frac{2(m-1)zA}{H}\frac{1+\left|D_{x'}f\right|^2}{f_z^{\frac32}}\\
	\leq&\frac{F+2f_t}{f_z}+2\sqrt{N}(m-1)zf_z^{-1}\sqrt{E}+C(m-1)zf_z^{-1}\left(1+\sqrt{E}\right)\\
	&+C(m-1)zBf_z^{-1}\left(\sqrt{E}+\sqrt{Q}\right).
\end{align*}
Young's inequality gives
\begin{align*}
	C(m-1)zf_z^{-1}\sqrt{E}\leq& \eps (m-1)zE+C_{\eps}(m-1)z f_z^{-2},\\
	C(m-1)zf_z^{-1}B\sqrt{Q}\leq&\eps(m-1)zBQ+C_{\eps}(m-1)zf_z^{-2}.
\end{align*}
Therefore,
\begin{align*}
	\label{esti Lz}
	\tag{4.16}
	\begin{split}
		\mathcal{L}z\leq \frac{F+2f_t}{f_z}+\eps(m-1)z\left(E+BQ\right)+C(m-1)z\left(f_z^{-1}+f_z^{-2}\right).
	\end{split}
\end{align*}

For the last term $\Lambda\left(\log\vartheta+\log\chi+z,\log\vartheta+\log\chi+z\right)$,
\begin{align*}
	\label{esti Lam}
	\tag{4.17}
	\begin{split}
	&\Lambda\left(\log\vartheta+\log\chi+z,\log\vartheta+\log\chi+z\right)\\
	=&(m-1)z\left[\left|D\left(\log\chi\right)-f_z^{-1}D_{x'}f\right|^2+\left|f_z^{-1}+D\left(\log\vartheta\right)\right|^2\right]\\
	\leq&(m-1)z\left[2\left|D\left(\log\chi\right)\right|^2+\left(2+2L^2\right)f_z^{-2}+2\left(D\left(\log\vartheta\right)\right)^2\right].
	\end{split}
\end{align*}

With all these estimates, we begin to prove our central estimates.
\begin{proof}[Proof of Theorem \ref{thm up grad}]
By the definition of $\mathcal{P}$, the positive maximum point of $\mathcal{P}$ on $\mathcal{D}$ lies on the top boundary or the interior. For the top boundary, our assumption gives $-f_t/H\leq H_1$. Hence we only consider the interior on which $\tau\chi\vartheta>0$. At the positive maximum point, we have  $f_t<0$. Since the maximum point of $\mathcal P$ is also the maximum point of $\log\mathcal{P}$, we get
\begin{align*}
	\log\mathcal{P}=&\log \tau+\log\chi +\log \vartheta+z+\log w,\\
	D\left(\log\mathcal{P}\right)=&D\left(\log \vartheta+\log\chi +z+\log w\right)=0,\\
	\left(\log \mathcal{P}\right)_t\geq&0,\qquad D^2\left(\log\mathcal{P}\right)\leq0.
\end{align*}
Since $\mathcal{A}$ is positive definite, we have
\begin{align*}
	\mathcal{L}\left(\mathcal{\log P}\right)\geq0.
\end{align*}
Notice that \eqref{product} and \eqref{Lw} give
\begin{align*}
	\mathcal{L}\left(\log w\right)=&T\left(\log w\right)-2\Lambda\left(\log H,\log w\right)\\
	=&\frac{Tw}{w}+\frac{\Lambda(w,w)}{w^2}-2\Lambda\left(\log H,\log w\right)\\
	=&\frac{1}{w}\left(-w\frac{TH}{H}+2\Lambda\left(\log H,w\right)\right)+\Lambda(\log w, \log w)-2\Lambda\left(\log H,\log w\right)\\
	=&-\frac{TH}{H}+\Lambda\left(\log w,\log w\right)\\
	=&-\frac{TH}{H}+\Lambda\left(\log \vartheta+\log\chi+z,\log \vartheta+\log\chi+z\right)
\end{align*}
Expanding $\mathcal{L}\left(\log\mathcal{P}\right)$ and using above identity,  \eqref{product}, \eqref{esti TH}, \eqref{computation2}-\eqref{esti Lam}, we obtain
\begin{align*}
	0\leq& \mathcal{L}\left(\log \tau\right)+ \mathcal{L}\left(\log \chi\right)+ \mathcal{L}\left(\log \vartheta\right)+ \mathcal{L}z+ \mathcal{L}\left(\log w\right)\\
	=&\tau^{-1}+ \mathcal{L}\left(\log \chi\right)+D\left(\log\vartheta\right) \mathcal{L}f-(m-1)zD^2\left(\log\vartheta\right)+ \mathcal{L}z\\
	&-\frac{TH}{H}+\Lambda\left(\log \vartheta+\log\chi+z,\log \vartheta+\log\chi+z\right)\\
	\leq&\tau^{-1}+\frac{F+2f_t}{f_z}-c_*\sqrt{f_z}\frac{f_t^2}{H}+(3\eps-c_*)(m-1)z\left(E+BQ\right)\\&
	+C(m-1)z\left(f_z^{-1}+f_z^{-2}\right)+Cf_z^{-\frac32}+C\left(f_z^{-1}-f_t\right)\left|D\left(\log\chi\right)\right|\\
	&+C(m-1)z\left[1+\left|D\left(\log\chi\right)\right|^2+\left|D^2\left(\log\chi\right)\right|+\left|D^2\left(\log\vartheta\right)\right|+\left(D\left(\log\vartheta\right)\right)^2+D\left(\log\vartheta\right)\right].
\end{align*}

Choose $\eps=c_*/4$, depending only on $m$, $n$, $L$, $C_H$, then
\begin{align*}
	(3\eps-c_*)(m-1)z\left(E+BQ\right)\leq 0.
\end{align*}

By assumption
\begin{align*}
	D\left(\log\vartheta\right)=&\frac{D\vartheta}{\vartheta}\leq C\vartheta^{-\frac12}\leq C\vartheta^{-1},&&\qquad  \left|D^2\left(\log\vartheta\right)\right|=\left|\frac{D^2\vartheta}{\vartheta}-\frac{\left(D\vartheta\right)^2}{\vartheta^2}\right|\leq C\vartheta^{-1},\\
	\left|D\left(\log \chi\right)\right|=&\frac{\left|D\chi\right|}{\chi}\leq C\chi^{-\frac12},&&\qquad \left|D^2\left(\log\chi\right)\right|=\left|\frac{D^2\chi}{\chi}-\frac{\left(D\chi\right)^2}{\chi^2}\right|\leq C\chi^{-1},
\end{align*}
where the constant $C$ depends additionally on $C_{\chi}$ and $C_{\vartheta}$ here and in the sequel.

Moreover, assumption gives $\Delta v\geq -K_{\delta}$. Combine with \eqref{pme v},
\begin{align*}
	v_t\geq |\nabla v|^2-(m-1)K_{\delta}v.
\end{align*}
Then, for $f$,
\begin{align*}
	(m-1)K_{\delta}zf_z^2-f_tf_z\geq 1+|D_{x'}f|^2.
\end{align*}
Solving this quadratic of $f_z$ gives
\begin{align*}
	f_z\geq \frac{f_t+\sqrt{f_t^2+4(m-1)K_{\delta}z\left(1+\left|D_{x'}f\right|^2\right)}}{2(m-1)K_{\delta}z}.
\end{align*}
Then
\begin{align*}
	F=\frac{1+\left|D_{x'}f\right|^2}{f_z}\leq& \frac{2(m-1)K_{\delta}z\left(1+\left|D_{x'}f\right|^2\right)}{f_t+\sqrt{f_t^2+4(m-1)K_{\delta}z\left(1+\left|D_{x'}f\right|^2\right)}}\\
	=&\frac{-f_t+\sqrt{f_t^2+4(m-1)K_{\delta}z\left(1+\left|D_{x'}f\right|^2\right)}}{2}.
\end{align*}
Using inequality $\sqrt{a^2+b}\leq a+\sqrt{b}$ for $a,b\geq0$ we obtain
\begin{align*}
	\label{ineq F}
	\tag{4.18}
	f_z^{-1}\leq F\leq -f_t+\sqrt{(m-1)K_{\delta}z\left(1+L^2\right)}=-f_t+C_F\sqrt{z},
\end{align*}
where $C_F=\sqrt{(m-1)K_{\delta}\left(1+L^2\right)}$. Therefore
\begin{align*}
	\frac{F+2f_t}{f_z}\leq f_tf_z^{-1}+C_F\sqrt{z}f_z^{-1}.
\end{align*}

Now we estimate, for $z\ll 1$, the terms
\begin{align*}
	Cf^{-\frac32}_z+C(m-1)z\left(f_z^{-1}+f_z^{-2}\right)+C_F\sqrt{z}f_z^{-1}.
\end{align*}
 If $f_z^{-1}\geq 2C_F$, then \eqref{ineq F} gives
\begin{align*}
	f_z^{-1}\leq -2f_t,\qquad f_z^{-2}\leq -2f_tf_z^{-1}.
\end{align*}
Choose $Q_0\geq\max\left\{1,2C_F\right\}$ large enough such that, for $f_z^{-1}\geq Q_0$, 
\begin{align*}
	Cf^{-\frac32}_z+C(m-1)zf_z^{-1}+C_F\sqrt{z}f_z^{-1}\leq \frac18f_z^{-2}\leq -\frac14f_tf_z^{-1}.
\end{align*}
For $f_z^{-1}<Q_0$,
\begin{align*}
Cf^{-\frac32}_z+C(m-1)zf_z^{-1}+C_F\sqrt{z}f_z^{-1}\leq C,
\end{align*}
where the constant $C$ depends additionally on $K_{\delta}$ here and in the sequel. Therefore, for all $f_z^{-1}>0$,
\begin{align*}
	Cf^{-\frac32}_z+C(m-1)zf_z^{-1}+C_F\sqrt{z}f_z^{-1}\leq -\frac14f_tf_z^{-1}+C.
\end{align*}
As for $C(m-1)zf_z^{-2}$, we can choose $Z_1$ small enough satisfying 
\begin{align*}
C(m-1)Z_1\leq \frac18,
\end{align*}
where $C$ is the coefficient of term $(m-1)zf_z^{-2}$. Then, for $f_z^{-1}\geq 2C_F$,
\begin{align*}
	C(m-1)zf_z^{-2}\leq \frac18 f_z^{-2}\leq -\frac14f_tf_z^{-1}.
\end{align*}
For $f_z^{-1}<2C_F$, we have $C(m-1)zf_z^{-2}\leq C$. Hence
\begin{align*}
	C(m-1)zf_z^{-2}\leq -\frac14f_tf_z^{-1}+C.
\end{align*}

Consequently, using all estimates we obtain
\begin{align*}
	\label{ineq central 1}
	\tag{4.19}
	c_*\frac{\sqrt{f_z}}{H}f_t^2-\frac12f_tf_z^{-1}\leq&\tau^{-1}+C\left(-f_t+1\right)\chi^{-\frac12}+C(m-1)z\left[1+\chi^{-1}+\vartheta^{-1}\right].
\end{align*}
Here, for $f_z^{-1}-f_t$, we use \eqref{ineq F} to show $f_z^{-1}-f_t\leq C(-f_t+1)$.

Now, we claim
\begin{align*}
	\label{ineq ft}
	\tag{4.20}
	\frac{\sqrt{f_z}}{H}f_t^2-f_tf_z^{-1}\geq C\left(-f_t\right)^{\frac53}\quad\text{for $-f_t\geq 1$}.
\end{align*}
In the case $f_z^{-1}\leq 1$,
\begin{align*}
	\frac{\sqrt{f_z}}{H}=\frac{1}{f_z^{-\frac12}\left(1+F_+-f+\frac{L^2-\left|D_{x'}f\right|^2}{2}\right)+z\left(1+9L^2-|D_{x'}f|^2\right)}\geq \frac{1}{2C_H},
\end{align*}
where $C_H$ is given in \eqref{bdd H}. Hence
\begin{align*}
	\frac{\sqrt{f_z}}{H}f_t^2-f_tf_z^{-1}\geq \frac{\sqrt{f_z}}{H}f_t^2\geq \frac{1}{2C_H}\left(-f_t\right)^{\frac53}.
\end{align*}
In the case $f_z^{-1}\geq1$, if $f_z^{-1}\geq \left(-f_t\right)^{\frac23}$, then
\begin{align*}
	\frac{\sqrt{f_z}}{H}f_t^2-f_tf_z^{-1}\geq-f_tf_z^{-1}\geq \left(-f_t\right)^{\frac53}.
\end{align*}
If $f_z^{-1}\leq \left(-f_t\right)^{{\frac23}}$, since $H\leq 2C_H$,
\begin{align*}
	\frac{\sqrt{f_z}}{H}f_t^2-f_tf_z^{-1}\geq\frac{\sqrt{f_z}}{H}f_t^2\geq \frac{1}{2C_H}\left(-f_t\right)^{\frac53}.
\end{align*}
Thus we prove the claim.

For $-Cf_t\chi^{-\frac12}$, Young's inequality gives
\begin{align*}
	-Cf_t\chi^{-\frac12}\leq \eps \left(-f_t\right)^{\frac53}+C_{\eps}\chi^{-\frac54}.
\end{align*}
Substitute above inequality and \eqref{ineq ft} into \eqref{ineq central 1}, and choose $\eps$ small enough such that $\eps \leq C/2$, where $C$ is given in \eqref{ineq ft}, then
\begin{align*}
	\left(-f_t\right)^{\frac53}\leq C\left(1+\chi^{-\frac54}+\vartheta^{-1}+\tau^{-1}\right)\quad\text{for $-f_t\geq1$}.
\end{align*}
When $-f_t\leq 1$, $-f_t/H$ is obviously bounded. Therefore, at the positive maximum point
\begin{align*}
	\left(-f_t\right)^{\frac53}\leq C\left(1+\chi^{-\frac54}+\vartheta^{-1}+\tau^{-1}\right).
\end{align*}
Multiply this inequality by $\left(\tau\chi e^z\vartheta/H\right)^{\frac53}$, since $H\geq 1$, $z,\chi,\vartheta\leq 1$
\begin{align*}
	\mathcal{P}^{\frac{5}{3}}\leq C\tau^{\frac53}\left(1+\chi^{\frac{5}{12}}+\vartheta^{\frac23}\right)+C\tau^{\frac{2}{3}}\leq C\left(\tau^{\frac53}+\tau^{\frac{2}{3}}\right).
\end{align*}
Finally, at the point on which $\chi=\vartheta=1$
\begin{align*}
\frac{-f_t}{H}=\frac{\mathcal{P}}{\tau e^z}	\leq C\left(1+\tau ^{-\frac35}\right).
\end{align*}
This makes the proof.
\end{proof}

With the central estimate in Theorem \ref{thm up grad} we give the upper gradient estimate.
\begin{corollary}\label{coro up grad}
	Under the hypotheses of Theorem \ref{thm up grad}. For $\gamma<z\leq Z$, $\chi=\vartheta=1$, $-1+\delta<t\leq 0$, there exists a new constant $C_1>0$, such that
	\begin{align*}
		\left|\nabla v\right|\leq C_1\left(1+H_1+\left(t+1-2\delta\right)^{-3/5}\right),\qquad|v_t|\leq C_1\left(1+H_1+\left(t+1-2\delta\right)^{-3/5}\right)^2.
	\end{align*}	
\end{corollary}
\begin{proof}
	For $f_t\leq 0$ and $f_z\leq 1$, we have $H\leq C+Cz\sqrt{f_z}\leq C$, then Theorem \ref{thm up grad} gives
	\begin{align*}
		-f_t\leq HC_1\left(1+H_1+\left(t+1-2\delta\right)^{-3/5}\right)\leq C_1\left(1+H_1+\left(t+1-2\delta\right)^{-3/5}\right).
	\end{align*}
Therefore, \eqref{ineq F} gives
\begin{align*}
	\left|\nabla v\right|=\frac{\sqrt{1+\left|D_{x'}f\right|^2}}{f_z}\leq C_1\left(1+H_1+\left(t+1-2\delta\right)^{-3/5}\right).
\end{align*}
Also
\begin{align*}
	\left|v_t\right|=\frac{-f_t}{f_z}\leq C_1\left(1+H_1+\left(t+1-2\delta\right)^{-3/5}\right)^2.
\end{align*}
Above estimates also holds for $f_z\geq 1$.

For $f_t\geq0$, we have $v_t\leq 0$, by our assumption 
\begin{align*}
	|\nabla v|^2\leq v_t+(m-1)K_{\delta}v\leq (m-1)K_{\delta}v\leq C_1.
\end{align*}
Also we have lower bound for $v_t$
\begin{align*}
	v_t\geq -(m-1)K_{\delta}v\geq C_1.
\end{align*}
It makes the proof.
\end{proof}



\section{Lower bound gradient estimates}\label{s5}
In this section, we establish the lower gradient estimate with the upper gradient estimates obtained in Section \ref{s4}. 

First, we show the linear growth of $v$ near the free boundary. Our proof of the following theorem is mainly based on the ideas in \cite{KZ01}.
\begin{theorem}\label{thm low grad}
	Let $v$ be the continuous weak solution of \eqref{pme v} on $B_{1-\delta}\times[-1+3\delta,0]$ on which
	\begin{align*}
		\left|\nabla v\right|\leq C_1,\qquad \left|v_t\right|\leq C_1,\qquad\Delta v \geq -K_{\delta}\quad\text{in the sense of distributions,}
	\end{align*} 
for a constant $C_1>0$ and $\delta$ is given in Lemma \ref{lem 3.1}, the constant $K_{\delta}>0$ is given in Proposition \ref{local AB prop}. Suppose the local cone condition holds for $v$ on $B_{1-\delta}\times[-1+3\delta,0]$. Consider a compact subset $\mathcal{K}$ of
\begin{align*}
	\left\{(x,t)\in B_{1-\delta}\times(-1+3\delta,0]:x\in\Gamma(t),\ x\notin\overline{\ome\left(-1+3\delta\right)}\right\}.
\end{align*}
Then there are constants $C_2>0$ and $\eps_*>0$, depend only on $n$, $m$, $\theta$, $C_1$, $K_{\delta}$, $\mathcal{K}$, such that for all $(q,t_0)\in\mathcal{K}$ and $0<\eps<\eps_*$
\begin{align*}
	v(q+\eps \mu,t_0)\geq C_2\eps.
\end{align*}
\end{theorem}

Before, we start to prove Theorem \ref{thm low grad}, we state some lemmas providing the construction of the supsolution.
\begin{lemma}\label{viscosity solution}
	Let $p(x)$ be a smooth function and $\psi(x)>0$ be a smooth function with $\left|\nabla \psi\right|<1/2$. Set
	\begin{align*}
		W(x)=\min_{|y-x|\leq \psi(x)}p(y).
	\end{align*}
Suppose that the minimizing point at $x_0$ is $y_0=x_0-\psi(x_0)e$ with $|e|=1$. If a smooth function $g$ touches $W$ from below at $x_0$, then
\begin{align*}
	\nabla g(x_0)=&\left|\nabla p(y_0)\right|\left(e-\nabla \psi (x_0)\right),\\
	\Delta g(x_0)\leq& \left|e-\nabla \psi(x_0)\right|^2\Delta p(y_0)+\left|\nabla p(y_0)\right|\left(-\Delta\psi(x_0)+(n-1)\frac{\left|\nabla \psi(x_0)\right|^2}{\psi(x_0)}\right).
\end{align*}
Suppose that the minimizing point lies in the interior, then
\begin{align*}
	\nabla g(x_0)=0,\qquad\Delta g(x_0)\leq 0.
\end{align*}
\end{lemma}
\begin{proof}
	For boundary minimum point, we construct a minimum point map $y(x)$ satisfying
	\begin{align*}
		y(x_0)=y_0,\qquad |y(x)-x|\leq \psi(x)\quad\text{for $x$ near $x_0$.}
	\end{align*}
Then we have
\begin{align*}
	p\left(y\left(x\right)\right)\geq W(x)\geq g(x),\qquad p\left(y\left(x_0\right)\right)=W(x_0)=g(x_0).
\end{align*}
By computing $\nabla p\left(y\left(x\right)\right)$ and $\Delta  p\left(y\left(x\right)\right)$ we makes the proof.

Choose an orthonormal basis with $e_n=e$ and $\nabla \psi(x_0)=a_1e_1+a_n e_n$, where $e_1\perp e_n$ and $a_1\geq 0$. Define an orthogonal matrix $\mathcal{A}$ by requiring that
\begin{align*}
	\mathcal{A}^T e=\frac{e-\nabla\psi(x_0)}{\left|e-\nabla\psi(x_0)\right|}.
\end{align*}
Above requirement shows that $\mathcal{A}$ is of form
\begin{align*}
	\mathcal A=
	\begin{pmatrix} 
		\frac{1-a_n}{\left|e-\nabla\psi(x_0)\right|}&0&&\frac{a_1}{\left|e-\nabla\psi(x_0)\right|}\\
		0&I_{n-2}&&0\\
		\frac{-a_1}{\left|e-\nabla\psi(x_0)\right|}&0&&\frac{1-a_n}{\left|e-\nabla\psi(x_0)\right|}
	\end{pmatrix}
.
\end{align*}
Define
\begin{align*}
	\mathcal{B}=\frac{I-e\otimes\nabla\psi(x_0)-|e-\nabla\psi(x_0)|\mathcal{A}}{\psi(x_0)}.
\end{align*}
Matrix $\mathcal{B}$ is of form
\begin{align*}
\mathcal{B}=\psi(x_0)^{-1}\begin{pmatrix} 
	a_n&0&&-a_1\\
	0&\left(1-\left|e-\nabla\psi(x_0)\right|\right)I_{n-2}&&0\\
	0&0&&0
\end{pmatrix}
.
\end{align*}
Hence
\begin{align*}
	\label{|B|}
	\tag{5.1}
	\begin{split}
	\left|\mathcal{B}\right|^2=&\frac{\left|\nabla\psi(x_0)\right|^2}{\left(\psi(x_0)\right)^2}+(n-2)\frac{\left(1-\left|e-\nabla\psi(x_0)\right|\right)^2}{\left(\psi(x_0)\right)^2}\\
	\leq&\frac{\left|\nabla\psi(x_0)\right|^2}{\left(\psi(x_0)\right)^2}+(n-2)\frac{\left|\nabla\psi(x_0)\right|^2}{\left(\psi(x_0)\right)^2}\\
	\leq &(n-1)\frac{\left|\nabla\psi(x_0)\right|^2}{\left(\psi(x_0)\right)^2}.
	\end{split}
\end{align*}
The requirement of $\mathcal{A}$ gives
\begin{align*}
	\label{eTB}
	\tag{5.2}
	\begin{split}
	e^T\mathcal{B}=&\psi(x_0)^{-1}\left[e^T-e^T\left(e\otimes\nabla \psi(x_0)\right)-|e-\nabla \psi(x_0)|e^T\mathcal{A}\right]\\
	=&\psi(x_0)^{-1}\left[e^T-\left(\nabla \psi(x_0)\right)^T-\left(e-\nabla \psi(x_0)\right)^T\right]\\
	=&0.
	\end{split}
\end{align*}

For $x$ near $x_0$, define
\begin{align*}
	E(x)=\frac{e+\mathcal{B}\left(x-x_0\right)}{\left|e+\mathcal{B}\left(x-x_0\right)\right|},\qquad y(x)=x-\psi(x)E(x).
\end{align*}
Using \eqref{eTB} we obtain
\begin{align*}
	E(x_0)=e,\qquad DE(x_0)=\mathcal{B}.
\end{align*}
Then for $x$ near $x_0$
\begin{align*}
	\label{y}
	\tag{5.3}
	\begin{split}
	y(x_0)=&x_0-\psi(x_0)e=y_0,\qquad\left|y(x)-x\right|=\psi(x)\\
	 D y(x_0)=&I-e\otimes \nabla \psi(x_0)-\psi(x_0)\mathcal{B}=\left|e-\nabla\psi(x_0)\right|\mathcal{A}.
	 \end{split}
\end{align*}
Since $\nabla p(y_0)=\left|\nabla p(y_0)\right|e$, \eqref{y} gives
\begin{align*}
	\nabla p\left(y(x_0)\right)=\left(Dy(x_0)\right)^T\nabla p(y_0)=\left|\nabla p(y_0)\right|\left|e-\nabla\psi(x_0)\right|\mathcal{A}^Te=\left|\nabla p(y_0)\right|\left(e-\nabla \psi(x_0)\right).
\end{align*}
For Laplacian,
\begin{align*}
\Delta y=-\Delta\psi E-2\nabla\psi\cdot\nabla E-\psi\Delta E.
\end{align*}
Since $E^2=1$, we have $\partial_i (E^2)=0$ and $\Delta (E^2)=0$ for any $i=1,\cdots,n$. Then, at $x=x_0$, we obtain $e\cdot \partial_i E(x_0)=0$ and $e\cdot \Delta E(x_0)=-\left|\mathcal{B}\right|^2$ for any $i=1,\cdots,n$. Hence
\begin{align*}
	\label{eDeltay}
	\tag{5.4}
	\begin{split}
	e\cdot\Delta y(x_0)=&-\Delta\psi(x_0)\left(e\cdot E(x_0)\right)-2e\cdot\left(\nabla \psi(x_0)\cdot \nabla E(x_0)\right)-\psi(x_0)e\cdot \Delta E(x_0)\\
	=&-\Delta\psi(x_0)+\psi(x_0)\left|\mathcal{B}\right|^2.
	\end{split}
\end{align*}
Therefore, \eqref{|B|}, \eqref{y} and \eqref{eDeltay} give
\begin{align*}
	\Delta p\left(y(x_0)\right)=&\sum_{i,k,l}\partial_i y_k(x_0)\partial_i y_l(x_0) \partial_{kl}p(y_0)+\sum_{i,k}\partial_{ii}y_k(x_0)\partial_k p(y_0)\\
	=&\left|e-\nabla \psi(x_0)\right|^2\Delta p(y_0)+\left|\nabla p(y_0)\right|\left(-\Delta\psi(x_0)+\psi(x_0)|\mathcal B|^2\right)\\
	\leq &\left|e-\nabla \psi(x_0)\right|^2\Delta p(y_0)+\left|\nabla p(y_0)\right|\left(-\Delta\psi(x_0)+(n-1)\frac{\left|\nabla \psi(x_0)\right|^2}{\psi(x_0)}\right).
\end{align*}

For interior minimum point, $y(x)$ stays admissible for $x$ near $x_0$, then the conclusion is obvious.
\end{proof}

\begin{lemma}\label{supsolution}
	Let $\mathcal{C}\Subset B_{1-\delta}$ be a bounded smooth domain and let $\phi(x)>0$ be a smooth function on $\overline{\mathcal{C}}$. Define
	\begin{align*}
		C_{\phi}=&\max_{\overline{\mathcal{C}}}\left\{|\nabla \phi|, -\Delta\phi +(n-1)\frac{\left|\nabla\phi\right|^2}{\phi}\right\}\\
		A_0=&1+2C_{\phi}+C_{\phi}^2, \qquad B_0=1+2(m-1)K_{\delta}A_0+(m-1)C_1C_{\phi}.
	\end{align*}
For any $r>0$ and $\tau>0$, choose $\eps>0$ such that
	\begin{align*}
		\eps\ll 1,\qquad \eps C_{\phi}\leq 1/2,\qquad B_0\eps \tau\leq 1/2.
	\end{align*}
Set
\begin{align*}
	g(s)=&\frac{1}{1-B_0\eps s},\qquad \zeta(0)=0,\qquad D\zeta(s)=\left(1+A_0\eps\right)g(s)\\
	w(x,s)=&g(s)\min_{|y-x|\leq \eps\phi(x)} v\left(y+\eps r \mu,\zeta(s)\right).
\end{align*}
Then $\left(\frac{m-1}{m}w\right)^{1/(m-1)}$ is the continuous weak supersolution of \eqref{pme u} on $\mathcal{C}\times\left[0,\tau\right]$.
\end{lemma}
\begin{proof}
	First, approximate $v$ with a sequence of the strictly positive smooth function $v_j$ with $|\nabla v_j|\leq C_1$, $\left|(v_j)_t\right|\leq C_1$, $\Delta v_j\geq -K_{\delta}$. Define $w_j$ by the same formula. Then we show $w_j$ is the supersolution.
	
	For the boundary minimum point, apply Lemma \ref{viscosity solution} with $\psi= \eps\phi$. Then in the sense of distributions
	\begin{align*}
		\label{w space}
		\tag{5.5}
		\nabla w_j=g\left|\nabla v_j\right|\left(e-\eps\nabla\phi\right),\qquad \Delta w_j\leq g\left|e-\eps \nabla \phi\right|^2\Delta v_j+g\left|\nabla v_j\right|\eps C_{\phi}.
	\end{align*}
Here, although $w_j$ is only Lipschitz, we can modify $w_j$ and pass to the limit in the integral inequality. To obtain the time derivative, use the minimum point map $y(x)$ in Lemma \ref{viscosity solution} and keep it fixed as time varies. The smooth function $g(s)v_j \left( y(x)+\eps r\mu, \zeta(s)\right)$ lies above $w_j$ and agrees at contact point.  Derivatives gives in the sense of distributions
\begin{align*}
	\label{w time}
	\tag{5.6}
 (w_j)_s=Dg v_j+gD\zeta(v_j)_t.
\end{align*}
Using \eqref{pme v}, \eqref{w space} and \eqref{w time}, we obtain
\begin{align*}
	&(w_j)_s-(m-1)w_j\Delta w_j-\left|\nabla w_j\right|^2\\
	\geq&Dgv_j+gD\zeta(v_j)_t-(m-1)g^2v_j\left(|e-\eps\nabla \phi|^2\Delta v_j+\left|\nabla v_j\right|\eps C_{\phi}\right)-g^2\left(e-\eps\nabla\phi\right)^2\left|\nabla v_j\right|^2\\
	=&Dgv_j+g\left(D\zeta-g\left|e-\eps\nabla\phi\right|^2\right)\left[(m-1)v_j\Delta v_j+\left|\nabla v_j\right|^2\right]-(m-1)g^2v_j\left|\nabla v_j\right|\eps C_{\phi}.
\end{align*}
Since $\left|e-\eps\nabla\phi\right|^2= 1+2\eps e\cdot \nabla\phi +\eps^2 \left|\nabla\phi\right|^2\in\left[1-A_0\eps,1+A_0\eps\right] $, we have
\begin{align*}
	D\zeta-g\left|e-\eps\nabla\phi\right|^2=g\left(1+A_0\eps-\left|e-\eps\nabla\phi\right|^2\right)\in\left[0,2gA_0\eps\right].
\end{align*}
Therefore, using $\left|\nabla v_j\right|\leq C_1$, $\Delta v_j\geq -K_{\delta}$ and $Dg=B_0\eps g^2$,
\begin{align*}
	&(w_j)_s-(m-1)w_j\Delta w_j-\left|\nabla w_j\right|^2\\
	\geq&B_0\eps g^2v_j-2(m-1)\eps g^2A_0v_jK_{\delta}-(m-1)\eps g^2v_jC_1 C_{\phi}\\
	=&\eps g^2 v_j\left[B_0-2(m-1)A_0K_{\delta}-(m-1)C_1C_{\phi}\right]\\
	=&\eps g^2 v_j\geq 0.
\end{align*}

For the interior minimum point, at $x=y_0$, we have
\begin{align*}
	(v_j)_t=(m-1)v_j\Delta v_j+\left|\nabla v_j\right|^2\geq0.
\end{align*}
Hence Lemma \ref{viscosity solution} gives in the sense of distributions
\begin{align*}
	(w_j)_s-(m-1)v_j\Delta v_j-\left|\nabla v_j\right|^2\geq Dg v_j+g D\zeta (v_j)_t\geq 0.
\end{align*}
Consequently, $w_j$ is the supsolution.

Finally, since $\left|\nabla w_j\right|\leq gC_1\left(1+\eps C_{\phi}\right)\leq C$, we have
\begin{align*}
	\left|\nabla\left( w_j^{\frac{m}{m-1}}\right)\right|\leq  Cw_j^{1/(m-1)}\left|\nabla w_j\right|\leq C.
\end{align*}
Then the weak compactness, weak convergence and uniform convergence show $\left(\frac{m-1}{m}w\right)^{\frac{1}{m-1}}$ is the continuous weak supersolution of \eqref{pme u}.
\end{proof}

Now, we construct a supsolution and use Proposition \ref{CF prop} to prove Theorem \ref{thm low grad}.
\begin{proof}[Proof of Theorem \ref{thm low grad}]
	For fixed compact subset $\mathcal{K}$, choose $\tau>0$ small enough such that for any $(q,t_0)\in\mathcal{K}$, $t_0-\tau\in \left[-1+3\delta,0\right]$ and
	\begin{align*}
		\label{tau}
		\tag{5.7}
		(m-1)K_{\delta}\tau\leq \frac14.
	\end{align*} 
	For fixed $(q,t_0)\in\mathcal{K}$ and above $\tau$, Theorem \ref{CF thm} gives a $r_q$ such that
	\begin{align*}
		v(x,t_0-\tau)\equiv0\quad\text{on $B_{r_q}(q)$}.
	\end{align*}
	The compactness of $\mathcal{K}$ gives a common $r>0$ such that for any $(q,t_0)\in\mathcal{K}$
	\begin{align*}
		v(x,t_0-\tau)\equiv0 \quad\text{on $B_{r}(q)$}.
	\end{align*}
	Also, we choose $r$ small enough such that $B_{r}(q)\subset B_{1-\delta}$ for any $(q,t_0)\in\mathcal{K}$.

	With above setting, now fix $(q,t_0)\in\mathcal K$ and translate it to $(0,\tau)$. Set
	\begin{align*}
		x^*=-\frac{r}{5}\mu,\qquad r^*=\frac{r\sin \theta}{10},\qquad\mathcal{C}= B_{r/2}(x^*)\backslash B_{r^*}(x^*).
	\end{align*}
	Obviously, origin $0\in \mathcal{C}$. Define
	\begin{align*}
		\alpha=\frac{\log\left(16/\sin \theta\right)}{\log (5/2)},\qquad \phi(x)=\frac{r\sin\theta}{4}\left(\frac{r}{2|x-x^*|}\right)^{\alpha}.
	\end{align*}
	On $\mathcal{C}$, function $\phi$ is smooth and positive. Moreover, it satisfies
	\begin{align*}
		\label{phi}
		\tag{5.8}
		\phi=&\frac{r\sin\theta}{4}\quad\text{on $\partial B_{r/2}(x^*)$},\qquad \phi(0)=4r,\qquad |\phi|\leq \frac{\sin \theta}{4}\left(\frac{5}{\sin\theta}\right)^{\alpha}r
	\end{align*}
	Choose $\eps>0$ satisfying
	\begin{align*}
		\label{range eps 1}
		\tag{5.9}
		\eps\leq \frac1{20},\qquad \eps\leq \frac{\sin\theta}{4},\qquad
		\eps\left(1+\frac{\sin \theta}{4}\left(\frac{5}{\sin\theta}\right)^{\alpha}\right)\leq \frac14.
	\end{align*}
	Then define
	\begin{align*}
		w(x,s)=g(s)\min_{|y-x|\leq \eps\phi(x)}v\left(y+\eps r \mu,\zeta(s)\right),
	\end{align*}
	where $g$, $\zeta$ are given in Lemma \ref{supsolution} and $\eps$ also satisfy the range in Lemma \ref{supsolution}. By \eqref{range eps 1}, for all $x\in\mathcal{C}$
	\begin{align*}
		|x|\leq \frac{7}{10}r,\qquad\left|y+\eps r\mu\right|\leq \left|y-x\right|+|x|+\eps r\leq \frac{19}{20}r.
	\end{align*}
	Therefore, $y+\eps r\mu\in B_{r}$ and we still have all estimates in our assumptions. Recall that $D\zeta\geq 2$ and $\zeta(0)=0$, then there exists a unique $0<s_{\tau}<\tau$ such that $\zeta(s_{\tau})=\tau$. Then Lemma \ref{supsolution} gives that $w$ is the supsolution on $\mathcal{C}\times[0,s_{\tau}]$. To use comparison principle, we check the order $w\geq v$ on the parabolic boundary.
	
	For initial time $s=0$, since $v(x,0)\equiv0$ on $B_r$, we know $w\geq0=v$.
	
	For the inner lateral boundary $(x,s)\in \partial B_{r^*}(x^*)\times[0,s_{\tau}]$, since $r^*< \sin\theta |x^*|$, we have  $-x\in \mathcal{C}_{\mu,\theta}$. Then local cone condition gives
	\begin{align*}
		v(x,s_{\tau})\leq v(0,s_{\tau})=0\quad\text{for $x\in\partial B_{r^*}(x^*)$}.
	\end{align*}
	Then Proposition \ref{local AB prop} implies above inequality holds for all $s\in[0,s_{\tau}]$. Therefore, $w\geq 0= v$ on $\partial B_{r^*}(x^*)\times[0,s_{\tau}]$.
	
	For the outer lateral boundary $(x,s)\in\partial B_{r/2}(x^*)\times[0,s_{\tau}]$, \eqref{phi} gives
	\begin{align*}
		\eps\phi=\frac{\sin\theta}{4}\eps r\leq \sin\theta \eps r|\mu|.
	\end{align*}
	Hence, $y+\eps r\mu-x\in\mathcal{C}_{\mu,\theta}$ for all $x\in\partial B_{r/2}(x^*)$ and $|y-x|\leq \eps\phi$ and then
	\begin{align*}
		\min_{|y-x|\leq \eps\phi(x)}v\left(y+\eps r \mu,\zeta(s)\right)\geq v\left(x,\zeta(s)\right)\quad\text{for $x\in\partial B_{r/2}(x^*)$}.
	\end{align*}
	Since $D\zeta\geq 2$ and $\zeta(0)=0$, we have $\zeta(s)>s$. Proposition \ref{local AB prop} then gives
	\begin{align*}
		v\left(x,\zeta(s)\right)\geq e^{-(m-1)K_{\delta}\left(\zeta(s)-s\right)} v(x,s).
	\end{align*}
	Combining these inequalities we have
	\begin{align*}
		w(x,s)\geq g(s)e^{-(m-1)K_{\delta}\left(\zeta(s)-s\right)} v(x,s)\quad\text{on $\partial B_{r/2}(x^*)\times[0,s_{\tau}]$}.
	\end{align*}
	Now we claim that $g(s)e^{-(m-1)K_{\delta}\left(\zeta(s)-s\right)}\geq 1$, for $s\in[0,s_{\tau}]$. Define
	\begin{align*}
		\Phi(s)=\log g(s)-(m-1)K_{\delta}\left(\zeta(s)-s\right).
	\end{align*}
	Then using the definition in Lemma \ref{supsolution} and \eqref{tau} we have
	\begin{align*}
		D\Phi=&\frac{Dg}{g}-(m-1)K_{\delta}\left(D\zeta-1\right)\\
		=&B_0\eps g-(m-1)K_{\delta} \left(A_0\eps g+B_0\eps sg\right)\\
		=&\eps g\left[B_0-(m-1)K_{\delta}\left(A_0+B_0s\right)\right]\\
		\geq&\eps g\left[\frac{3}{4}B_0-(m-1)K_{\delta}A_0\right]\geq0.
	\end{align*}
	Since $\Phi(0)=0$, we prove the claim. Therefore, $w\geq v$ on $\partial B_{r/2}(x^*)\times[0,s_{\tau}]$.
	
	Since we have checked that $w\geq v$ on the parabolic boundary, then we obtain $w\geq v$ on $\mathcal{C}\times[0,s_{\tau}]$.  Consider $s=s_{\tau}$ and $x\in B_{\eps r}$. Then \eqref{range eps 1} gives that
	\begin{align*}
		\eps r\leq \frac{1}{20}r\leq \frac{1}{2}r-|x^*|=\frac{3}{10}r,\qquad \eps r\leq \frac{1}{20}r\leq \frac{1}{10}r\leq |x^*|-r^*.
	\end{align*}
	Hence $B_{\eps r}\subset\mathcal{C}$. Moreover, for $x\in B_{\eps r}$, assumptions in Lemma \ref{supsolution} and \eqref{phi} gives
	\begin{align*}
		\phi(x)\geq \phi(0)- C_{\phi}|x|\geq 4r- C_{\phi}\eps r> 3r.
	\end{align*}
	At this moment, consider $y^*=-\eps r\mu$, then
	\begin{align*}
		|y^*-x|=|x+\eps r \mu|\leq 2\eps r \leq \eps \phi(x).
	\end{align*}
	Therefore,
	\begin{align*}
		\label{initial v=0}
		\tag{5.10}
		0=g(s_{\tau})v\left(0,\tau \right)=g(s_\tau )v\left(y^*+\eps r \mu,\zeta(s_{\tau})\right)\geq w(x,s_{\tau})\geq v(x,s_{\tau})\quad\text{on $x\in B_{\eps r}$}.
	\end{align*}

	Now, we estimates $\tau-s_{\tau}$. Write
	\begin{align*}
		\tau-s_{\tau}=\int_{0}^{s_{\tau}}\left(D\zeta-1\right)ds=\int_{0}^{s_{\tau}}\eps g\left(A_0+B_0s\right)ds.
	\end{align*}
	Since $g(s)\leq 2$ for $s\in[0,s_{\tau}]$, we obtain
	\begin{align*}
		\tau-s_{\tau}\leq  2\tau \left(A_0+B_0\tau\right)\eps.
	\end{align*}
	
	Finally, we apply Proposition \ref{CF prop} to complete the proof with $x_0=0$, $t_0=s_{\tau}$, $R=\eps r$, $\sigma=\tau-s_{\tau}$. Here, we additionally require
	\begin{align*}
		2\tau \left(A_0+B_0\tau\right)\eps\leq \sigma_0.
	\end{align*}
	Then Proposition \ref{CF prop} and \eqref{initial v=0} gives
	\begin{align*}
		\frac{1}{|B_{\eps r}|}\int_{B_{\eps r}} v(x,\tau) dx> \frac{c_0 \eps ^2 r^2}{\tau-s_{\tau}}\geq C_2\eps.
	\end{align*}
	Otherwise, we must have $v(x,\tau)\equiv0$ on $B_{\eps r/6}$, which is a contradiction. Then the local cone condition and \eqref{range eps 1} give that
	\begin{align*}
		v\left(\frac{2}{\sin \theta} r \eps \mu,\tau \right)\geq \frac{1}{|B_{\eps r}|}\int_{B_{\eps r}} v(x,\tau) dx\geq C_2\eps.
	\end{align*}
	Since $\eps$ is arbitrarily small, we finish the proof with a new constant $C_2$.
	\end{proof}

Since we have prove the linear growth near free boundary, a standard rescaling argument combining with Harnack inequality in \cite{CVW87} gives the lower gradient estimate near free boundary. We give the theorem in the following and omit the proof. More details see the proof of Theorem $5$ in \cite{CVW87}.
\begin{theorem}
	\label{thm low grad 2}
	Under the hypotheses of Theorem \ref{thm low grad}. There are constants $C_2>0$ and $\delta_*>0$, depend only on $n$, $m$, $\theta$, $C_1$, $K_{\delta}$, $\mathcal{K}$, such that for all $(q,t_0)\in\mathcal{K}$ we have
	\begin{align*}
		\left|\nabla v\right|\geq C_2\quad\text{on $\left(B_{\delta_*}(q)\times[t_0-\delta_*,t_0]\right)\cap \Omega$}.
	\end{align*}
\end{theorem}



\section{Proof of main theorems}\label{s6}
In this section, we prove our main theorems and give the counterexample. 

\begin{proof}[Proof of Theorem \ref{thm 1.1}]
	For any $(q,t_0)\in\Gamma\cap \left(B_1\times(-1,0]\right)$ with $q\notin \overline{\ome(-1)}$, choose $\delta>0$ satisfying $q\in B_{1-2\delta}$ and $-1+4\delta\leq t_0$. Then Proposition \ref{local AB prop} gives that for any $\delta_0\leq \delta$,
	\begin{align*}
		\Delta v\geq -K_{\delta}\quad \text{in the sense of distributions $\quad$on $B_{\delta_0}(q)\times[t_0-\delta_0,t_0]$}.
	\end{align*}
	
First, we can use Proposition \ref{smooth appro} and Proposition \ref{local AB prop} to obtain a sequence of smooth $v_j$ with $\Delta v_j \geq -K_{\delta}$ on $B_{\delta_0}(q)\times[t_0-\delta_0,t_0]$ for sufficiently small $\delta_0$. Indeed, we can choose suitable $B_0^{n-1}$ and $F_-$ such that $B_{\delta}(q)\subset\left\{x'\in B_0^{n-1},\ F_-\leq x_n\leq f(x',Z_0,t)\right\}$, where $Z_0$ is give in Lemma \ref{lem 3.1}. Then choose $\delta_0$ small such that $B_{\delta_0}(q)\times[t_0-\delta_0,t_0]$ belongs to the domain in Proposition \ref{smooth appro}, and that Proposition \ref{local AB prop} can be applied.

Second, for such $v_j$ we want to apply Corollary \ref{coro up grad} to obtain the upper gradient estimate. The first step is to check that $H_1<\infty$ in Theorem \ref{thm up grad}. Indeed, for $z=Z$, $v$ and $v_j$ is smooth with uniform bound independent of $j$. Thus we have $H_1<\infty$. Then, fix suitable domain $\mathcal{D}$ and cut-off functions $\chi$, $\vartheta$ in Theorem \ref{thm up grad}. Choose $\delta_0$ such that $B_{\delta_0}(q)\times[t_0-\delta_0,t_0]\subset \mathcal D$ and $\chi=\vartheta=1$ for $x\in B_{\delta_0}(q)$. Then Corollary \ref{coro up grad}  gives that
	\begin{align*}
		\left|\nabla v_j\right|\leq C_1,\qquad \left|(v_j)_t\right|\leq C_1\quad\text{on $\left(B_{\delta_0}(q)\times\left[t_0-\delta_0,t_0\right]\right)\cap\ome$}.
	\end{align*}
 Then taking a limit $v_j\to v$ gives
\begin{align*}
	\label{v up grad}
	\tag{6.1}
	\left|\nabla v\right|\leq C_1,\qquad \left|v_t\right|\leq C_1\quad\text{on $\left(B_{\delta_0}(q)\times\left[t_0-\delta_0,t_0\right]\right)\cap\ome$}.
\end{align*}

Third, fix a compact set $\mathcal K$ in Theorem \ref{thm low grad 2} satisfying $(q,t_0)\in \mathcal{K}$. Choose $\delta_0$ such that $\delta_0\leq \delta_*$. Combining Theorem \ref{thm low grad 2} with \eqref{v up grad} gives
\begin{align*}
	\left|\nabla v\right|\geq C_2,\qquad D_{\mu}v\geq C_2\quad\text{on $\left(B_{\delta_0}(q)\times\left[t_0-\delta_0,t_0\right]\right)\cap\ome$}.
\end{align*}

Finally, applying Theorem \ref{thm tool} we make the proof with a new smaller $\delta_0$.
\end{proof}

Then we start to prove Theorem \ref{short-time}. We first prove the weak non-degeneracy condition \eqref{wndc} forces the free boundary to move immediately.
\begin{proposition}
	\label{ndc force move}
	Suppose $v_0$ satisfies weak non-degeneracy condition \eqref{wndc}, and that $\ome(0)$ satisfies the inner cone condition, namely, for every boundary point $X\in\Gamma(0)$, there exists an open cone $\mathcal{C}_{X}$ with vertex $X$ such that $\ov{\mathcal{C}_X}\subset \ov{\ome(0)}$. Then
	\begin{align*}
		\overline{\ome(0)}\Subset \ome(t) \quad\text{for every $t>0$.}
	\end{align*}
\end{proposition}
\begin{proof}
	Define the waiting time $t^*(X)$ at some boundary point $X\in\Gamma(0)$ by
	\begin{align*}
		t^*(X)=\sup\left\{s:v(X,t)=0\quad\text{for $0\leq t\leq s$}\right\}.
	\end{align*}
	Alikakos \cite{A85} showed that 
	\begin{align*}
		t^*(X)>0\quad\text{if and only if}\quad \limsup_{r\to0}\left|B_{r}(X)\right|^{-\frac{2}{(m-1)n}-1}\int_{B_r(X)}u_0(x)<\infty.
	\end{align*}
	Condition \eqref{wndc} and inner cone condition give that at every boundary point $X$
	\begin{align*}
		\limsup_{r\to0}\left|B_{r}(X)\right|^{-\frac{2}{(m-1)n}-1}\int_{B_r(X)}u_0(x)\geq C r^{\frac{\beta-2}{m-1}}\to \infty.
	\end{align*}
	Therefore, $t^*(X)=0$ for every boundary point $X$.
	\end{proof}
	Now we construct the quadratic solution of \eqref{pme v}. Since the proof is trivial, we omit here.
	\begin{lemma}
		\label{quadratic solution}
		Set 
		\begin{align*}
			Q(x,t)=a(t)|x|^2+b(t)\cdot x+c(t),
		\end{align*}
		where
		\begin{align*}
			a'=\left[2n(m-1)+4\right]a^2,\qquad b'=\left[2n(m-1)+4\right]ab,\qquad c'=2n(m-1)ac+|b|^2.
		\end{align*}
		Then $\left(Q(x,t)\right)_+$ is the weak solution of \eqref{pme v} as long as $a$, $b$, $c$ is well-defined.
	\end{lemma}
	Now we prove Theorem \ref{short-time}.
	\begin{proof}[Proof of Theorem \ref{short-time}]
		First, approximate $v_0$ with a sequence of compactly supported function that is smooth up to the boundary. In fact, we can extend $v_0\in C^1\left(\ov{\ome(0)}\right)$ to $F_0\in C^1\left(\R^n\right)$ with
		\begin{align*}
		F_0=v_0 \quad\text{on $\ov{\ome(0)}$},\qquad F_0<0 \quad\text{outside $\ov{\ome(0)}$},\qquad F_0=\text{constant}\quad\text{outside a large ball.}
		\end{align*}
		Then by modifying $F_0$ we obtain a sequence of smooth function $F_{0j}\in C^{\infty}(\R^n)$. Also $F_{0j}\to F_{0}$ in $C^1$. Setting $v_{0j}=\left(F_{0j}\right)_+$, we finish the approximation. For fixed $j$, Theorem \ref{DH thm} gives that there exists a time $T_j>0$ such that the smooth solution exists on $(0,T_j)$ with initial data $v_{0j}$. Denote this solution by $v_j$ and $\ome_{j}(t)=\left\{x:v_{j}(x,t)>0\right\}$, $\Gamma_j(t)=\partial \ome_{j}(t)$.

		Our assumptions give that we can choose a constant $\rho>0$ such that for large $j$
		\begin{align*}
			\label{distance}
			\tag{6.2}
			C_2\leq\frac{v_{0j}(x)}{\operatorname{dist}(x,\Gamma_j(0))}\leq C_1\quad\text{on $\ome_j(0)\cap U:=\left\{x: \operatorname{dist}(x,\Gamma(0))\leq 2\rho\right\}$}
		\end{align*}
		for constant $C_1,C_2>0$, depending only on $\rho$, $c$, $\|\nabla v_0\|_{L^{\infty}}$. Using \eqref{distance}, we can easily construct a supsolution to show
		\begin{align*}
			\label{boundary upper speed}
			\tag{6.3}
			\ov{\ome_j(t)}\subset\left\{x:\operatorname{dist}(x,\ov{\ome_j(0)})\leq \left[4n(m-1)+8\right]C_1t\right\}\quad\text{for small $t>0$}.
		\end{align*}
		Indeed, \eqref{distance} gives 
		\begin{align*}
			v_{0j}(x)\leq C_1\left(|x-X|-d_j(X)\right)\quad\text{for some $X\notin \ov{\ome_j(0)}$},
		\end{align*}
		 where $d_j(X)=\operatorname{dist}\left(X,\ov{\ome_j(0)}\right)>0$. Then comparison with supsolution
		\begin{align*}
			A(t)\left(|x-X|-d_j(X)/2\right)^2_+,\qquad A(0)=2C_1/d_j(X),\qquad A'=\left[2n(m-1)+4\right]A^2
		\end{align*}
		gives \eqref{boundary upper speed} for all $\left[4n(m-1)+8\right]C_1t<d_j(X)$. With the upper speed of the free boundary \eqref{boundary upper speed}, there exists a time $T_0>0$, such that $\Gamma_j(t)\Subset U$ for  $t\in[0,T_0]$ and large $j$.
		
		Second, we define $\theta_0=\pi/12$, $e_0(x)=\nabla F_0(x)/|\nabla F_0(x)|$ and
		\begin{align*}
			J\left(x,p(x,t)\right)=e_0(x)\cdot p(x,t)-|p(x,t)|\cos \theta_0.
		\end{align*}
		Assume $J\left(X,\nabla v_j(X,T)\right)>0$, then we have $\angle\left(\nabla v_j(X,T),e_0(X)\right)<\theta_0$. Then in a small ball centered at $X$ on which $\angle\left(e_0(x),e_0(X)\right)<\theta_0$, we obtain $v_j$ is non-decreasing along the cone $\mathcal{C}_{e_0(X),\pi/4}$, i.e. the local cone condition. Notice that the choice of the small ball is possible because $F_0\in C^1\left(\R^n\right)$.
		
		When $t=0$, let $p=\nabla F_0$, recall that $|\nabla F_{0}|\geq C_2$ on $U$, then
		\begin{align*}
			 J(x,\nabla F_0)=|\nabla F_0|\left(1-\cos\theta_0\right)\geq \left(1-\cos\theta_0\right)C_2>0\quad\text{on $U$}.
		\end{align*}
		Thus, for large $j$, $J\left(x,\nabla v_{0j}\right)>0$ on $\ov{\ome_j(0)}\cap U$. Since $v_j$ is smooth up to the free boundary for $t\in(0,T_j)$, $J(x,\nabla v_j)$ is continuous on $\ov{\ome_j(t)}\cap U$ for $t\in\left[0,T_j\right)$. Define
		\begin{align*}
			S_j:=\sup\left\{T>0: v_j(x,t) \text{ is smooth up to $\Gamma_j(t)$ and} \min_{x\in\ov{\ome_j(t)}\cap U}J(x,\nabla v_j)>0\text{ on $(0,T)$}\right\}.
		\end{align*}
		For fixed $j$, previous arguments show $S_j>0$. Then we want to show $S_j$ has uniform positive lower bound. If so, taking a limit we obtain the local cone condition of true solution $v$ in a small time interval. Combining with Proposition \ref{ndc force move}, Theorem \ref{thm 1.1} gives the short-time smoothness.
		
		Third,we use blow-up analysis to show $S_j$ has uniform positive lower bound. Otherwise, there exists a subsequence of $j$ such that $S_j\to0$. Here and in the following, we use same notation $j$ to represent the subsequence of $j$.
		
		For fixed $j$, we claim that at $t=S_j$, we must have $\min_{x\in\ov{\ome_j(S_j)}\cap U}J(x,\nabla v_j)=0$. Since $\min_{x\in\ov{\ome_j(t)}\cap U}J(x,\nabla v_j)>0$ on $[0,S_j)$, previous arguments show the local cone condition holds on $[0,S_j)$. By continuity, the local cone condition also holds on $[0,S_j]$. Then Theorem \ref{thm 1.1} and Proposition \ref{ndc force move} show that $v_j$ is smooth up to the free boundary at $t=S_j$ with non-degeneracy condition. Theorem \ref{DH thm} gives a $T_j>0$ such that $v_j$ is smooth up to the free boundary on $(0,S_j+T_j)$. Hence if $\min_{x\in\ov{\ome_j(S_j)}\cap U}J(x,\nabla v_j)>0$, by the smoothness we just obtained, $\min_{x\in\ov{\ome_j(t)}\cap U}J(x,\nabla v_j)>0$ holds in a short time interval starting at $t=S_j$. This contradicts the maximality of $S_j$. Thus the claim is proved. 
		
		Notice that $J(x,\nabla v_j)=0$ infers that either $|\nabla v_j|=0$ or $\angle\left(e_0(x),\nabla v_j\right)=\theta_0$. Since $\min_{x\in\ov{\ome_j(S_j)}\cap U}J(x,\nabla v_j)=0$, for fixed small $\delta>0$, we can choose $t_j<S_j$ and $x_j\in \ov{\ome_j(t_j)}\cap U$ such that
		\begin{align*}
			\label{either or}
			\tag{6.4}
			\text{either}\quad\left|\nabla v_j(x_j,t_j)\right|\leq \delta\quad\text{or}\quad \angle\left(e_0(x_j),\nabla v_j(x_j,t_j)\right)\geq \theta_0-\delta.
		\end{align*}

		Denote $d_j=\operatorname{dist}(x_j,\Gamma_j(0))$, consider two cases, in the sense of subsequence, $d_j/t_j$ is uniformly bounded and $d_j/t_j\to\infty$.
 	
 	For the case $1$, $d_j/t_j$ is uniformly bounded, choose a boundary point $P_j\in \Gamma_j(0)$ such that $\left|x_j-P_j\right|=d_j$ and define $Y_j=(x_j-P_j)/t_j$. Since $|Y_j|=d_j/t_j$ is uniformly bounded, after passing to subsequence we obtain $P_j\to P_{\infty}\in\Gamma(0)$ and $Y_j\to Y_{\infty}\in\R^n$. Rotate the coordinates such that $\nabla F_0(P_{\infty})=\lam e_n$ with constant $\lam\geq C_2$.
 	
 	 Define blow-up sequence
 	\begin{align*}
 		V_j(y,s)=t_j^{-1}v_j\left(P_j+t_jy,t_js\right)\quad\text{on $\R^n\times[0,1]$}.
 	\end{align*}
 	It's easy to verify $V_j$ satisfies \eqref{pme v}. Consider the initial time $s=0$, Taylor expansion gives
 	\begin{align*}
 		\phi_j(y):=t_j^{-1}F_{0j}(P_j+t_jy)\to \lam y_n \quad\text{locally uniformly on $\R^n$}.
 	\end{align*}
 	Since $\phi_j(0)=0$ and $\left|\nabla \phi_j\right|\leq C_1$, for any small $\eps>0$, we have for large $j$ depending on $\eps$
 	\begin{align*}
 		 \label{F0j estimate}
 		\tag{6.5}
 		\left|\phi_j(y)-\lam y_n\right|\leq \eps\left(1+|y|^2\right)\quad\text{on $\R^n$}.
 	\end{align*}
 	Then consider $Q^{\pm}(y,0)=\lam y_n\pm \eps\left(1+|y|^2\right)$, Lemma \ref{quadratic solution} shows the corresponding solution $\left(Q^{\pm}(y,s)\right)_+$ is the weak solution of \eqref{pme v} on $\R^n\times[0,1]$, if $\eps$ is small enough. Then comparison principle and \eqref{F0j estimate} give for large $j$ depending on $\eps$
 	\begin{align*}
 		\left(Q^-(y,s)\right)_+\leq V_j\leq \left(Q^+(y,s)\right)_+\quad\text{on $\R^n\times[0,1]$}.
 	\end{align*}
 	Observing that when $\eps\to 0$, $Q^{\pm}(y,s)\to \lam(y_n+\lam s)$ locally uniformly on $\R^n\times[0,1]$. Thus after letting $j\to \infty$ first and then $\eps\to 0$
 	\begin{align*}
 		V_j\to V_{\infty}=\lam\left(y_n+\lam s\right)_+\quad\text{locally uniformly on $\R^n\times[0,1]$}.
 	\end{align*}
 	Moreover, on the compact subset of $\left\{y_n+\lam s <0,\ 0\leq s\leq 1\right\}$, we have $Q^+<0$ for small $\eps>0$. Fixing such $\eps$ and choose $j$ large enough we have 
 	\begin{align*}
 		\label{V_j=0}
 		\tag{6.6}
 		V_j\leq \left(Q^+\right)_+=0 \quad\text{on above compact subset}.
 	\end{align*}
 	
 	If $V_{\infty}\left(Y_{\infty},1\right)>0$, the interior estimates give $\nabla V_j\left(Y_j,1\right)\to \lam e_n$, hence $\left|\nabla v_j\left(x_j,t_j\right)\right| \geq \lam/2$ and $\angle\left(e_0(x_j),\nabla v_j\left(x_j,t_j\right)\right)\leq \theta_0/2$. This contradicts \eqref{either or}. 
 	
 	If $Y_{\infty,n}<-\lam$, then $V_j(x,1)\equiv0$ near $Y_j$ for large $j$ by \eqref{V_j=0}. While $x_j\in\ov{\ome_j(t_j)}\cap U$, this is impossible. 
 	
 	If $Y_{\infty,n}=-\lam$, since $V_j$ preserves the local cone condition near $Y_{j}$, Proposition \ref{ndc force move} and the proof of Theorem \ref{thm 1.1} give the smoothness of $V_j$ near $Y_{j}$ with uniform derivatives estimates. Thus, $\nabla V_j(Y_j,1)\to\lam e_n$, which contradicts \eqref{either or}. Therefore, the case $1$ is impossible.
 	
 	For the case $2$, $d_j/t_j\to \infty$ in the sense of subsequence, we show $x_j\in\ome_j(0)\cap U$ for large $j$. Otherwise $x_j\notin \ome_j(0)$ but $x_j\in\ov{\ome_j(t_j)}\cap U$. Then \eqref{boundary upper speed} gives $d_j/t_j$ is uniformly bounded, a contradiction. Compactness gives $x_j\to x_{\infty}\in\ov{\ome(0)}\cap U$ in the sense of subsequence. Rotate the coordinates such that $\nabla F_0\left(x_{\infty}\right)=\lam e_n$ with constant $\lam\geq C_2$.
 	
 	Set $p_j=v_{0j}(x_j)>0$ and $\delta_j=\sqrt{t_j/p_j}$. Then by \eqref{distance}, $p_j/t_j\to\infty$, thus $\delta_j\to 0$. Define
 	\begin{align*}
 		V_j(y,s)=p_j^{-1}v_j\left(x_j+\sqrt{p_jt_j}y,t_js\right).
 	\end{align*}
 	$V_j$ satisfies \eqref{pme v}. At the initial time $s=0$, Taylor expansion gives
 	\begin{align*}
 		\phi_j(y):=\frac{F_{0j}\left(x_j+\sqrt{p_jt_j}y\right)-p_j}{\sqrt{p_jt_j}}\to \lam y_n\quad\text{locally uniformly on $\R^n$.}
 	\end{align*}
 	Since $\phi_j(0)=0$ and $|\nabla \phi_j|\leq C_1$, for any small $\eps>0$, we have for large $j$ depending on $\eps$
 	\begin{align*}
 		\label{phij esti}
 		\tag{6.7}
 	\left|\phi_j(y)-\lam y_n\right|\leq \eps\left(1+|y|^2\right)\quad\text{on $\R^n$}.
 	\end{align*}
 	Recall that \eqref{distance} shows $\sqrt{p_jt_j}/d_j\leq C_1\sqrt{t_j/d_j}\to 0$. Hence $B_{C\sqrt{p_jt_j}}(x_j)\subset \ome_j(0)$ for any finite constant $C>0$. So
 	\begin{align*}
 		v_j\left(x_j+\sqrt{p_jt_j}y,0\right)=F_{0j}(x_j+\sqrt{p_jt_j}y) \quad \text{on $\R^n$.}
 	\end{align*}
 	 Then
 	\begin{align*}
 		V_j(y,0)=p_j^{-1}v_{j}\left(x_j+\sqrt{p_jt_j}y,0\right)=1+\delta_j\phi_j(y) \quad\text{on $\R^n$.}
 	\end{align*}
 	\eqref{phij esti} shows for large $j$ depending on $\eps$
 	\begin{align*}
 		1+\delta_j\left[\lam y_n-\eps\left(1+|y|^2\right)\right]\leq V_j(y,0)\leq 1+\delta_j\left[\lam y_n+\eps\left(1+|y|^2\right)\right]\quad\text{on $\R^n$}.
 	\end{align*}
 	Similarly, using Lemma \ref{quadratic solution} and comparison principle and letting $j\to\infty$ first and then $\eps\to0$, we obtain
 	\begin{align*}
 		V_j\to 1,\qquad W_j=\frac{V_j-1}{\delta_j}\to \lam y_n\quad\text{locally uniformly on $\R^n\times[0,1]$}.
 	\end{align*}
 	
 	Compute the equation of $W_j$
 	\begin{align*}
 		\left(W_j\right)_s=(m-1)V_j\Delta W_j+\nabla V_j\cdot\nabla W_j.
 	\end{align*}
 	Since $1/2\leq V_j\leq 2$ for large $j$, the interior estimates give that $\nabla W_j(0,1)\to \lam e_n$. Hence $\nabla v_j(x_j,t_j)\to \lam e_n$, which contradicts \eqref{either or}. Therefore, the case $2$ is impossible and we finish the proof.
	\end{proof}

Then we start to prove Theorem \ref{global-time}. We assume that $\ome(0)$ is convex and prove a standard monotonicity property 
that implies the local cone condition. We can then apply Theorem \ref{thm 1.1} to 
complete the proof of Theorem \ref{global-time}. Since a similar monotonicity property has been proved in \cite{DHL01}, we omit the proof here. More details see the proof of Proposition $3.9$ in \cite{DHL01}.

For every point $x\in\R^n$, let $d(x)=\operatorname{dist}(x,\Omega(0))$ and $y\in\overline{\ome(0)}$ be the point such that $|y-x|=d(x)$. Since $\ome(0)$ is convex, $y$ is unique for each $x$. Define
\begin{align*}
	e(x):=\frac{x-y}{|x-y|}\quad\text{for $x\notin \overline{\ome(0)}$}.
\end{align*}
Then we have the following monotonicity property.

\begin{proposition}
	\label{monoto}
	Suppose $\Omega(0)$ is convex. Then, for any set $U\Subset\left(\overline{\ome(0)}\right)^c$, there exists a constant $\theta\in(0,\pi/2]$, depending only on $U$ and $\Omega(0)$, such that if unit vector $e$ satisfies
	\begin{align*}
		\angle\left(e,e(x)\right)\leq \theta,
	\end{align*}
	then $v(x+\eps e,t)\leq v(x,t)$ for all $\eps>0$ small, $x\in U$ and $t\geq0$.
\end{proposition}

Now we prove Theorem \ref{global-time}.
\begin{proof}[Proof of Theorem \ref{global-time}]
	With condition \eqref{wndc}, Proposition \ref{ndc force move}	gives that for any $(q,t_0)\in\Gamma\cap \left(\R^n\times(0,\infty)\right)$, we obtain $q\notin \overline{\ome(0)}$. Then Proposition \ref{monoto} gives a local cone $\mathcal C_{-e(x),\theta}$ on which $v$ is non-decreasing. Since $\ome(0)$ is convex, it follows that $e(x)$ is Lipschitz continuous away from $\partial \ome(0)$. Thus, $v$ satisfies local cone condition at the point $(q,t_0)$. After translation and scaling, Theorem \ref{thm 1.1} and Theorem \ref{CF thm} give the desired conclusion.
\end{proof}

To the end of this paper, we give a counterexample that shows the conditions of Theorem \ref{short-time} and Theorem \ref{global-time} are almost optimal in some sense.
\begin{proposition}
	\label{counterexample}
	For every $n\geq 2$ and $m>1$, there exists a bound connected star-shaped Lipschitz domain $\ome(0)\subset \R^n$ and a nonnegative function $v_0\in C^1\left(\ome(0)\right)\cap C^{0,1}\left(\R^n\right)$ supported on $\ome(0)$ such that
	\begin{align*}
		v_0(x)\geq \frac{1}{2}\operatorname{dist}\left(x,\Gamma(0)\right),\qquad v_0+\left|\nabla v_0\right|\geq \frac{1}{12}\quad\text{on $\ome(0)$}.
	\end{align*}
	Let $v$ be the weak solution of \eqref{pme v} with initial data $v_0$. Then $\ov{\ome(0)}\Subset \ome(t)$ for every $t>0$. Nevertheless, there exists $\tau_j\downarrow0$ and $X_j\in\Gamma(\tau_j)\to 0$ such that $\Gamma(\tau_j)$ is not $C^1$ hypersurface in any neighborhood of $X_j$.
\end{proposition}
\begin{remark}
	When $n=1$, Angenent \cite{A88} proved that if $\ome(0)$ is a bounded interval, then, after the waiting time, the free boundary is analytic. The condition in Proposition \ref{counterexample} implies the waiting time is zero for all dimension $n\geq 1$. If $n=1$, the short-time smoothness holds, even the global-in-time smoothness. While for $n\geq 2$, it fails by Proposition \ref{counterexample}. On the other hand, the interval is convex, and Theorem \ref{global-time} is consistent with the result in \cite{A88}.
\end{remark}
\begin{proof}
	Proposition \ref{ndc force move} immediately gives $\ov{\ome(0)}\Subset\ome(t)$ for every $t>0$ for above initial data.

	Then, we construct the initial support and initial data. Frist, denote
	\begin{align*}
		K=2n(m-1)+4,\qquad H=100\cdot 32^{K/2},\qquad \Lambda=100H,\qquad \lam=20H,\qquad\gamma_0=1/(5H).
	\end{align*}
	Let
	\begin{align*}
		\mathcal{C}=\left\{(x',x_n)\in \R^{n-1}\times\R: x_n\geq 0,\ |x'|\leq \gamma_0 x_n\right\},\qquad \ome(0)= B_1\backslash \mathcal{C}.
	\end{align*}
Obviously, $\ome(0)$ is a bounded connected star-shaped Lipschitz domain. For $x\in \ome(0)$, let
\begin{align*}
	d(x)=\operatorname{dist}\left(x,\mathcal{C}\right),\qquad F(x)=\frac{d(x)\left(1-|x|\right)}{d(x)+1-|x|}.
\end{align*}
Set $F=0$ outside $\ome(0)$, then $\operatorname{supp} F=\ome(0)$. It's easy to verify that $F\in C^1\left(\ome(0)\right)\cap C^{0,1}\left(\R^n\right)$ and
\begin{align*}
	\label{F prop}
	\tag{6.8}
	|x|\geq F(x)\geq \frac{1}{2} \dist\left(x,\Gamma(0)\right),\qquad F+|\nabla F|\geq \frac{1}{12}\quad\text{on $\ome(0)$}.
\end{align*}

We finish the construction of the basic part and continue to construct the multiscale amplifier factor. Denote
\begin{align*}
	\beta_0=4\log H+1,\qquad \sigma=e^{-\beta_0-4},\qquad r_1=\frac{e^{-\beta_0}}{8\Lambda},\qquad r_j=r_1\sigma^{j-1}\quad\text{for $j\geq 1$}.
\end{align*}
We claim there is a smooth function $G: \R\to[1,H]$, with bound derivatives, such that
\begin{align*}
	\label{G property}
	\tag{6.9}
	\begin{split}
		&G'\geq -G/2,\\
		&G\left(\log |x|\right)=1 \quad\text{for all $|x|\geq 1/8$ and $\lam r_j/2\leq |x|\leq \Lambda r_j/2$},\\
		&G\left(\log |x|\right)=H \quad\text{for all $0.99\Lambda r_j\leq |x|\leq 1.01 \Lambda r_j$}.
	\end{split}
\end{align*}
In fact, take a smooth nondecreasing cut-off function $\chi$ with $\chi=0$ on $\left(-\infty,0\right]$, $\chi=1$ on $\left[1,\infty\right)$, and $0\leq \chi'\leq 2$. Define
\begin{align*}
	g(y)=\log H\chi\left(2y+\log 4\right)\chi \left(\frac{\beta_0-y}{4\log H}\right),\qquad G(y)=\exp\left(\sum_{j=1}^{\infty}g\left(y-\log\left(\Lambda r_j\right)\right)\right).
\end{align*}
Since $g\left(y-\log\left(\Lambda r_j\right)\right)$ supports on $\left(\log\left(\Lambda r_j\right)-\log2, \log\left(\Lambda r_j\right)+\beta_0\right)$, for each $j$, the supports are disjoint. Thus, $G\in[1,H]$ is smooth with bound derivatives. Moreover, computations gives $G$ satisfies \eqref{G property}.

Define
\begin{align*}
	v_0(x)=G\left(\log|x|\right)F(x)\quad\text{on $\R^n$}.
\end{align*}
Then we see 
\begin{align*}
	v_0\in C^1\left(\ome(0)\right)\cap C^{0,1}\left(\R^n\right)),\qquad \operatorname{supp} v_0=\ome(0),\qquad v(x)\geq \frac{1}{2}\operatorname{dist}\left(x,\Gamma(0)\right)\quad\text{on $\ome(0)$}.
\end{align*}
Moreover, computations give that $v_0+|\nabla v_0|\geq 1/12$ on $\ome(0)$. Therefore, $v_0$ satisfies our conditions.

Second, for any point $y\in\R^n$, constant $r,a>0$, define
\begin{align*}
	\begin{split}
	&	w_{y,r,a}(x,t)=\frac{a}{1+Ka t}\left[r^2\left(1+Ka t\right)^\frac{4}{K}-|x-y|^2\right]_+,\\
	&W_{y,r,a}(x,t)=\frac{a}{1-Ka t}\left[|x-y|^2-r^2\left(1-Ka t\right)^{\frac{4}{K}}\right]_+.
	\end{split}
\end{align*}
These two explicit solution are standard barriers for \eqref{pme v}. Set $q_j=\left(0,\lam r_j\right)\in \R^{n-1}\times\R$. Then we show
\begin{align*}
	\label{up barrier}
	\tag{6.10}
	v_0(x)\leq W_{q_j,r_j,1/r_j}(x,0)\quad\text{on $\R^n$}.
\end{align*}
It suffices to consider $x\in\ome(0)$. If $|x|\leq \lam r_j/2$, the inequality $G'/G\geq -1/2$ shows that $|x|G\left(\log |x|\right)$ is increasing with respect to $|x|$. Hence by \eqref{F prop}, \eqref{G property} and $\lam=20H\geq 20$
\begin{align*}
	v_0\leq |x|G\left(\log |x|\right)\leq \lam r_j/2\leq W_{q_j,r_j,1/r_j}(x,0)\quad\text{on $\ome(0)$}.
\end{align*}
If $\lam r_j/2\leq |x|\leq \Lambda r_j/2$, then $v_0(x)=F(x)\leq d(x)\leq |x-q_j|$. While
\begin{align*}
	|x-q_j|\geq \dist\left(q_j,\partial \mathcal{C}\right)=\frac{\gamma_0\lam r_j}{\sqrt{1+\gamma_0^2}}>3r_j.
\end{align*}
Hence
\begin{align*}
	W_{q_j,r_j,1/r_j}(x,0)\geq |x-q_j|\geq v_0(x)\quad\text{on $\ome(0)$}.
\end{align*}
 If $\Lambda r_j/2\leq |x|\leq 1$, then $|x-q_j|\geq 3|x|/5$. Hence
\begin{align*}
	W_{q_j,r_j,1/r_j}(x,0)\geq \frac{\left(3|x|/5\right)^2-r_j^2}{r_j}\geq \frac{|x|^2}{50r_j}\geq H|x|\geq v_0(x)\quad\text{on $\ome(0)$}.
\end{align*}
Therefore, we obtain \eqref{up barrier}. Comparison principle gives
\begin{align*}
	\label{up control}
	\tag{6.11}
	v(x,t)=0\quad\text{on $B_{2^{-\frac{2}{K}}r_j}(q_j)\times \left[0,\frac{r_j}{2K}\right]$}.
\end{align*}

Fix a unit vector $e'\in\R^{n-1}$ and let $z_j=\left(24r_je', \Lambda r_j\right)$. For $x\in B_{r_j}(z_j)$, we have
\begin{align*}
	0.99\Lambda r_j< |x| < 1.01 \Lambda r_j,\qquad d(x)>2r_j.
\end{align*}
Indeed, $d(z_j)=\left(24r_j-\gamma_0\Lambda r_j\right)/\sqrt{1+\gamma_0^2}=4r_j/\sqrt{1+\gamma_0^2}>3r_j$, then $d(x)>2r_j$ on $B_{r_j}(z_j)$. Therefore, $F(x)>d/2>r_j$ and $v_0(x)>Hr_j$ on $B_{r_j}(z_j)$. Moreover,
\begin{align*}
	w_{z_j,r_j,H/(2r_j)}(x,0)\leq Hr_j/2\leq v_0(x) \quad\text{on $B_{r_j}(z_j)$},\qquad w_{z_j,r_j,H/(2r_j)}=0 \quad\text{outside $B_{r_j}(z_j)$}.
\end{align*}
By comparison principle, when $t=2r_j\left(32^{K/2}-1\right)/\left(KH\right)=:T_j$,
\begin{align*}
	v(x,T_j)\geq w_{z_j,r_j,H/(2r_j)}(x,T_j)>0\quad\text{on $B_{32r_j}(z_j)$}.
\end{align*}
In particular,
\begin{align*}
	\label{lower control}
	\tag{6.12}
	v\left(\left(0,\Lambda r_j\right),T_j\right)>0.
\end{align*}
Notice that
\begin{align*}
	T_j=\frac{2r_j\left(32^{K/2}-1\right)}{KH}<\frac{r_j}{50K}<\frac{r_j}{2K}.
\end{align*}

Third, for fixed $j\geq 2$, define
\begin{align*}
	I_j=\left\{(0,x_n): \lam r_j\leq x_n\leq \lam r_{j-1}\right\},\qquad \tau_j=\inf\left\{t>0: I_j\cap \ome(t)\neq \varnothing\right\}.
\end{align*}
Since $I_j\Subset \mathcal{C}$, the finite propagation property, also see \eqref{boundary upper speed}, gives $\tau_j>0$. Also since $\lam r_j<\Lambda r_j<\lam r_{j-1}$, \eqref{lower control} gives $\tau_j< T_j$.  

At $t=\tau_j$, definition gives $v(x,\tau_j)=0$ on $I_j$. Also we know $I_j\cap \Gamma(\tau_j)$ is nonempty. Otherwise, the finite propagation property shows for some $t_j>\tau_j$, we still have $I_j\cap \ome(t_j)=\varnothing$. This is a contradiction. Moreover, since $T_j<r_j/(2K)<r_{j-1}/(2K)$, \eqref{up control} gives the points in $I_j\cap \Gamma(\tau_j)$ lie at a positive distance from the two endpoints of $I_j$. Let $X_j$ be the point in $I_j\cap \Gamma(\tau_j)$ with the least $x_n$ coordinate. 

By our construction, the initial data, and hence the solution, is invariant under orthogonal transformation of $x'$. Suppose $\Gamma(\tau_j)$ were a $C^1$ hypersurface near $X_j$. When $n\geq 3$, the symmetry forces the normal of $\Gamma(\tau_j)$ at $X_j$ to be parallel to the $x_n$ axis. Then the free boundary is locally represented by $x_n=g(x')$. While for small $\eps_n>0$, consider two points $X_j^+=X_j+(0,\eps_n)$ and $X_j^-=X_j-(0,\eps_n)$, we have $v(X_j^+,\tau_j)=v(X_j^-,\tau_j)=0$. This is impossible by our previous assumptions.

For $n=2$, the symmetry leaves two cases. One is that the tangent is perpendicular to the $x_2$ axis, which is discussed in previous arguments. If the tangent is parallel to the $x_2$ axis, write the free boundary as $x_1=g(x_2)$ near $X_j$. Symmetry implies $g=-g$, so $g=0$ near $X_j$. While we have chosen $X_j$ to be the point with the least $x_2$ coordinate. Hence this is impossible. Consequently, $X_j$ must be the singularity point at $t=\tau_j$.

Letting $j\to \infty$, we obtain a sequence of singular times $\tau_j\to0$ and singular points $X_j\to0$. Thus we finish the proof.
\end{proof}


\end{document}